\documentclass{amsart}

\usepackage{amssymb, amsmath}
\usepackage{mathrsfs}
\usepackage{amscd}
\usepackage[active]{srcltx}
\usepackage{verbatim}

\usepackage[colorlinks,linkcolor={blue},citecolor={blue},urlcolor={red},]{hyperref}

\usepackage[fleqn,tbtags]{mathtools}

\usepackage{newsymbol}

\let\mathcal \undefined
\def\mathcal{\mathscr}

\let\emptyset \undefined
\let\ge       \undefined
\let\le       \undefined
\newsymbol\le          1336  \let\leq\le
\newsymbol\ge          133E  \let\geq\ge
\newsymbol\emptyset    203F
\newsymbol\notle       230A
\newsymbol\notge       230B

\theoremstyle{plain}
\newtheorem{theorem}{Theorem}[section]
\theoremstyle{remark}
\newtheorem{remark}[theorem]{Remark}
\newtheorem{example}[theorem]{Example}
\theoremstyle{plain}
\newtheorem{corollary}[theorem]{Corollary}
\newtheorem{lemma}[theorem]{Lemma}
\newtheorem{proposition}[theorem]{Proposition}
\newtheorem{definition}[theorem]{Definition}

\numberwithin{equation}{section}

\def\N{{\mathbb N}}
\def\Z{{\mathbb Z}}

\def\R{{\mathbb R}}
\def\C{{\mathbb C}}

\newcommand{\E}{{\mathbb E}}
\renewcommand{\P}{{\mathbb P}}

\def\om{\omega}

\newcommand{\eps}{\varepsilon}
\newcommand{\calL}{\mathscr{L}}
\newcommand{\n}{\Vert}
\newcommand{\one}{{{\bf 1}}}

\newcommand{\s}{^*}
\newcommand{\lb}{\langle}
\newcommand{\rb}{\rangle}
\newcommand{\limn}{\lim_{n\to\infty}}
\newcommand{\limk}{\lim_{k\to\infty}}

\newcommand{\sumnN}{\sum_{n=1}^N}

\newcommand{\Dom}{{\mathsf{D}}}
\newcommand{\Ran}{{\mathsf{R}}}
\newcommand{\Ker}{{\mathsf{N}}}

\newcommand{\wt}{\widetilde}
\newcommand{\wh}{\widehat}
\newcommand{\dd}{{\rm d}}
\newcommand{\ud}{\,{\rm d}}

\newcommand{\D}{\Dom}

\allowdisplaybreaks 

\begin{document}

 \title[The Weyl calculus for group generators satisfying CCR]
 {The Weyl calculus for group generators satisfying the canonical commutation relations}

\author{Jan van Neerven}

\address{Delft Institute of Applied Mathematics\\
Delft University of Technology\\P.O. Box 5031\\2600 GA Delft\\The Netherlands}
\email{J.M.A.M.vanNeerven@tudelft.nl}

\author{Pierre Portal}
\address{The Australian National University, Mathematical Sciences Institute, John Dedman Building, 
Acton ACT 0200, Australia.}
\email{Pierre.Portal@anu.edu.au}     

\date{\today}


\thanks{The authors gratefully acknowledge financial support by the ARC Discovery Project DP160100941.}

\begin{abstract}
Classical pseudo-differential calculus on $\R^{d}$ can be viewed as a (non-commutative) functional calculus for the standard position and momentum operators $(Q_{1}, \dots , Q_{d})$ and $(P_{1}, \dots , P_{d})$. We generalise this calculus to the setting of two $d$-tuples of operators $A=(A_{1}, \dots , A_{d})$ and $B=(B_{1}, \dots , B_{d})$ acting on a Banach space $X$ such that $iA_{1}, \dots , iA_{d}$ and $iB_{1}, \dots , iB_{d}$ generate bounded $C_0$-groups 
satisfying the Weyl canonical commutation relations:
\begin{equation*}
\begin{aligned}
e^{isA_j}e^{itA_k} &= e^{itA_k}e^{isA_j},\quad e^{isB_j}e^{itB_k} = e^{itB_k}e^{isB_j}\\
      e^{isA_j}e^{itB_k} & = e^{-ist \delta_{jk}} e^{itB_k}e^{isA_j}.
\end{aligned}
\end{equation*}
We show that the resulting calculus $a\mapsto a(A,B) \in \calL(X)$, initially defined for Schwartz functions $a\in \mathscr{S}(\R^{2d})$, extends to symbols in the standard symbol class $S^{0}$ of pseudo-differential calculus provided appropriate bounds can be established.
We also prove a transference result that bounds the operators $a(A,B)$ in terms of the twisted convolution operators $C_{\wh a}$ acting on $L^{2}(\R^{2d};X)$. We apply these results to
obtain $R$-sectoriality and boundedness of the $H^{\infty}$-functional calculus (and even the H\"ormander calculus), for the abstract harmonic oscillator $L = \frac12\sum_{j=1}^d (A_j^2+B_j^2)-\frac12d$.
\bigskip

\begin{center}
{\em In memory of Alan McIntosh (1942-2016), celebrating his friendship with Jos\'e Enrique Moyal (1910-1998).}
\end{center}
\end{abstract}

\maketitle

\section{Introduction}

In the early 1980's, Alan McIntosh introduced the $H^{\infty}$-functional calculus as a refined version of the Dunford holomorphic functional calculus for unbounded sectorial operators (see the original paper \cite{McI} in the Hilbert space setting, their extensions to Banach spaces \cite{cdmy,KalWei}, and the monographs \cite{Haase, HNVW2}). This calculus is meant to be an operator-theoretic abstraction of the calculus of Fourier multipliers, which it recovers when applied to constant coefficient differential operators such as the Laplacian on $L^{2}(\R^{d})$. One of its roles is to provide a framework for perturbation theory: deriving properties of the functional calculus of differential operators with varying coefficients from its constant coefficient counterpart. The quintessential example of such an application is given in \cite{AKMc}, where perturbation is first understood in the operator-theoretic sense, then in the harmonic analytic sense of (an extension of) Calder\'on--Zygmund theory.
The combination of both perspectives leads to a striking boundedness result for the $H^{\infty}$-functional calculus of Dirac operators that includes the solution of the celebrated Kato's square root problem (originally obtained in \cite{AHLMcT}).\\

In the present paper, we introduce an operator-theoretic framework which aims to generalise pseudo-differential calculus in the same way that the McIntosh $H^{\infty}$-functional calculus generalises Fourier multiplier calculus. Our starting point is the Weyl calculus of standard position and momentum operators $Q_{j}f(x)=x_{j}f(x)$ and $P_{j}f(x) = i\partial_{j}f(x)$, $j=1,\dots,d$, acting on their natural domains in $L^{2}(\R^{d})$. For Schwartz functions $a \in \mathcal{S}(\R^{2d})$ one can define
a bounded operator $a(Q,P)$ acting on $L^2(\R^d)$ by  
\begin{equation}
\label{intro:def}
a(Q,P)f = \frac{1}{(2\pi)^{d}}
\int _{\R^{2d}} \widehat{a}(u,v) e^{i(uQ+vP)}f\ud u\ud v, \quad f\in L^2(\R^d).
\end{equation}
Here, $ e^{i(uQ+vP)}$ is understood as the Schr\"odinger representation 
\begin{equation}\label{intro:Schr} e^{i(uQ+vP)}f(x) := e^{\frac12iuv}e^{iuQ}e^{ivP}f(x) =
e^{\frac{1}{2}iuv + iux} f(x+v)
\end{equation}
which unitarily represents the canonical commutation relations for the position and momentum operators on $L^2(\R^d)$; the first identity is suggested by the Baker--Campbell-Hausdorff formula, noting that all higher commutators of $P$ and $Q$ vanish.
As shown in \cite[Proposition 1, page 554]{Stein}, \eqref{intro:def} encodes the standard pseudo-differential calculus, in the sense that for every $a\in \mathcal{S}(\R^{2d})$ there exists a unique $b \in \mathcal{S}(\R^{2d})$ such that
\begin{equation}
\label{intro:pseudodef}
a(Q,P)f(x) = \frac1{(2\pi)^{d/2}}\int _{\R^{d}} b(x,\xi)\widehat{f}(\xi)e^{i x\xi}\ud\xi,
\end{equation}
the map $a \mapsto b$ being continuous with respect to various relevant topologies. The advantage of \eqref{intro:def} over \eqref{intro:pseudodef} is that the former makes sense for generators of bounded groups on an arbitrary Banach spaces, whereas a representation such as \eqref{intro:pseudodef} is restricted to function spaces on which an appropriate Fourier transform can be defined.
We thus take \eqref{intro:def} as our starting point for the development of a general theory.\\

We work with general Weyl pairs (see Section \ref{sec:Weyl} for precise definition), i.e., two $d$-tuples 
$A=(A_{1}, \dots , A_{d})$ and $B=(B_{1}, \dots , B_{d})$ acting on a Banach space $X$ such that $iA_{1}, \dots , iA_{d}$ and $iB_{1}, \dots , iB_{d}$ generate bounded $C_0$-groups  satisfying the canonical (integrated) commutation relations
\begin{equation}
\begin{aligned}
e^{isA_j}e^{itA_k} &= e^{itA_k}e^{isA_j},\quad e^{isB_j}e^{itB_k} = e^{itB_k}e^{isB_j}\\
      e^{isA_j}e^{itB_k} & = e^{-ist \delta_{jk}} e^{itB_k}e^{isA_j}.
\end{aligned}
\end{equation}
In this context, \eqref{intro:Schr} is replaced by 
$$  e^{i(uA+vB)} := e^{\frac12iuv}e^{iuA}e^{ivB} := e^{\frac12iuv}\prod_{j=1}^d e^{iu_jA_j}\prod_{k=1}^d e^{iv_kB_k}.$$ 
The analogue of \eqref{intro:def},
\begin{equation}
\label{intro:def2}
a(A,B)f = \frac{1}{(2\pi)^{d}}
\int _{\R^{2d}} \widehat{a}(u,v) e^{i(uA+vB)}f\ud u\ud v, \quad f\in X,
\end{equation}
defines an algebra homomorphism between $\mathcal{S}(\R^{2d})$ endowed with the (non-comm\-ut\-ative) Moyal product
\begin{equation*}
\begin{aligned} \ & (a\,\#\,b)(x,\xi) \\ & \ \ = 
\frac{1}{\pi^{2d}} \int_{\R^{2d}}\int_{\R^{2d}} a(x+u, \xi+u') b(x+v, \xi+v') e^{-2i(vu'-uv')}\ud u\ud u'\ud v\ud v'
\end{aligned}
\end{equation*}
into the space of bounded linear operators $\calL(X)$. The Moyal product is used in pseudo-differential operator theory to deal with composition of symbols. In Section \ref{sec:Weyl}, we show that if this algebra homomorphism is continuous from $\mathcal{S}(\R^{2d})$ endowed with the topology of the standard pseudo-differential class of symbols $S^{0}$ to $\calL(X)$, then the calculus can be meaningfully extended from $\mathcal{S}(\R^{2d})$ to $S^{0}$. This is an analogue of the fundamental convergence lemma in the theory of $H^{\infty}$-functional calculus (see, e.g. \cite[Proposition 10.2.11]{HNVW2}), and is proved using asymptotic expansions of the Moyal product, typical of pseudo-differential calculus. Having such a convergence lemma shows that a pseudo-differential calculus for $(A,B)$ can be defined as soon as appropriate bounds on the operators defined in \eqref{intro:def2} are obtained.\\

One of the applications of pseudo-differential calculus is to study Schr\"odinger operators such as the harmonic oscillator defined by $\frac12\Delta f(x)- \frac12|x|^2f(x)$ on $L^{2}(\R^{d})$. In our abstract situation, we show that it is possible to express, in Section \ref{sec:A2B2}, the semigroup generated by 
\begin{equation}\label{eq:defL} -L:= \frac12d - \frac12\sum_{j=1}^d (A_j^{2}+B_j^{2})
\end{equation} 
in terms of the Weyl calculus as  
\begin{equation}\label{eq:exptL}e^{-tL} = a_t(A,B),
\end{equation}
 where $a_t\in \mathscr{S}(\R^{2d})$ is the function
\begin{align*} a_{t}(x,\xi) := \Big(1+\frac{1-e^{-t}}{1+e^{-t}}\Big)^{d}\exp\Bigl(-\frac{1-e^{-t}}{1+e^{-t}}(|x|^{2}+|\xi|^{2})\Bigr).
\end{align*}
For the pair of position and momentum operators associated with the Ornstein--Uhlenbeck operator (see Example \ref{ex:OU}), \eqref{eq:exptL} is a well known formula for the Orn\-stein--Uhlenbeck semigroup which goes back, at least, to \cite{unter}; see also \cite{NP}, where this formula was rediscovered by a reduction to Mehler's formula.
Here we show, with a different proof, that it generally holds for the operators $L$ associated with Weyl pairs through \eqref{eq:defL}. As such, \eqref{eq:exptL} can be thought of as an abstract analogue of Mehler's formula for Weyl pairs.

To obtain useful bounds for various functions of $L$ we use, in Section \ref{sec:transf}, the idea of transference to derive bounds for $a(A,B)$ acting on $X$ from corresponding bounds on the twisted convolution with $\wh a$, viewed as an operator acting on 
$L^{p}(\R^{2d};X)$. This idea can be traced back to Coifman and Weiss \cite{CW} and the form used here is inspired by the work of Hieber and Pr\"uss \cite{HiePru}, Haase \cite{H},
and Haase and Rozendaal \cite{HR}.
They have shown that bounds on the Phillips functional calculus defined, for a generator $iG$ of a bounded $C_0$-group acting on a Banach space $X$, by
$$
a(G)f = \frac{1}{\sqrt{2\pi}} 
\int _{\R} \widehat{a}(u)e^{iuG} f\ud u,
$$
can be obtained from bounds on convolution operators acting on $L^{2}(\R;X)$. The latter can then be proven using, for instance, Bourgain's UMD-valued Fourier multiplier theorem \cite{Bour86}, or its analogue for operator-valued kernels proven by Weis in \cite{Wei}.

For twisted convolutions, however, no UMD-valued theory is yet available. Developing such a theory 
is bound to be difficult, given that the (scalar-valued) $L^p$-theory of twisted convolutions, as developed by Mauceri in \cite{mauceri80}, is already subtle (see also \cite{MPR}). For applications to spectral multipliers theorems for $L$, fortunately, we only need to handle highly specific twisted convolutions that can effectively be ``untwisted''. This is shown in Section \ref{sec:untwist}, 
where we prove $R$-sectoriality 
for the operator $L$ defined by \eqref{eq:defL}   
in UMD lattices $X$. In Section \ref{sec:Hinfty},
we use this result to deduce the boundedness of the $H^\infty$-calculus of $L$ 
on UMD lattices $X$ from the boundedness of the Weyl calculus of $(A,B)$. We also show that the angle of this calculus is best possible (namely $0$). Going even further, we apply the recent Kriegler--Weis approach to spectral multipliers developed in \cite{Krieg,KriegW} to show that this $H^{\infty}$-calculus can in fact be extended to a H\"ormander class of sufficiently smooth but not necessarily analytic functions. This is possible because the estimates obtained in Section \ref{sec:untwist} are precise enough for us to check the assumption of \cite{KriegW}. \\

The present paper provides a foundation for a generalised pseudo-differential operator theory in at least three directions: Witten pseudo-differential calculus, global pseudo-differential calculus on Lie groups, and rough pseudo-differential calculus. 
In the Witten pseudo-differential calculus, one is interested in pairs $(A,B)$ acting on $L^{p}(\R^d,e^{-\phi(x)}\ud x)$, such that, informally, the ``Witten Laplacian'' $L$ is of the form $h(A,B)$ for an appropriate ``Hamiltonian''  $h$ which is chosen so that the measure $e^{-\phi(x)}\ud x$ is an invariant measure for $L$. 
We started such a theory in \cite{NP} in the most classical case where the choice $\phi(x)=\frac12|x|^{2}$ brings us back to the Gaussian setting and $L$ reduces to the Ornstein--Uhlenbeck operator. In work in progress, some of the results proven in the present paper 
are applied to extend the functional calculus theory of the Ornstein--Uhlenbeck operator in \cite{GMMST}. 

From the Lie group point of view, the present paper can be seen as an approach to (sub)pseudo-differential calculus on $L^p(H)$, where $H$ is the Heisenberg group. The prefix ``sub" here indicates that we consider a pseudo-differential calculus that extends the Fourier multiplier calculus given by the functional calculus of the sub-Laplacian (removing this prefix by extending the present paper to add $\partial_{t}$ to the joint functional calculus of the Weyl pair $(X,Y)$, in the spirit of \cite{stric}, would be interesting).
In the setting where $X$ is an $L^p$-space, a Lie group representation approach to some of the results in Section \ref{sec:A2B2} 
has already been pursued in \cite{DG,DGE} for more general higher-order commutator relations; see also 
\cite{ElstRob94, ElstRob98}. Building on earlier work in \cite{ElstRob94}, in the setting of $L^p$-spaces 
the boundedness of the $H^\infty$-calculus of $\eps+L$ for $\eps>0$ has been proved in \cite{Smul} by more direct transference arguments.
The present operator-theoretic perspective could help construct global pseudo-differential calculi on nilpotent Lie groups. 
Such a theory is currently being developed by Ruzhansky, Fischer, and their collaborators (see, in particular, \cite{FR}). The theory of pseudo-differential operators on the Heisenberg group is developed 
in \cite{BFKG}.

Last but not least, we aim to perturb the Weyl calculus, both from an operator-theoretic and a harmonic analytic perspective, to eventually treat pairs of the form
\begin{align*}
Q_{B,j}f(x) &= \frac{1}{2}\Bigl(\partial_{j}f(x)+x_{j}f(x) - \sum _{k=1} ^{d}\beta_{kj}(x)(\partial_{k}f(x)-x_{k}f(x))\Bigr), \\
P_{B,j}f(x) &= \frac{1}{2i}\Bigl(\partial_{j}f(x)+x_{j}f(x) + \sum _{k=1} ^{d}\beta_{kj}(x)(\partial_{k}f(x)-x_{k}f(x))\Bigr),
 \end{align*}
 where both the matrix $B=(\beta_{kj})_{k,j=1}^d$ and its inverse have bounded measurable coefficients. Notice that we recover the standard pair with $B=I$.
These are analogues of the perturbations of Dirac operators considered in \cite{AKMc}. Since the latter can be interpreted as a
rough Fourier multiplier theory, a corresponding theory for $(Q_{B},P_{B})$ could be interpreted as a rough pseudo-differential calculus.\\

\noindent {\em Acknowledgment} 
We thank Tom ter Elst, Markus Haase, Sean Harris, and Javier Parcet for interesting discussions. We dedicate this paper to the memory of Alan McIntosh (1942-2016). His philosophy of using operator theory as a mean to extend harmonic analysis towards rougher settings, very much underpins the present research. 
Alan McIntosh was a close friend of Joe Moyal, whose phase space perspective on quantum mechanics gives the non-commutative structure on appropriate algebra of functions that we use here to extend McIntosh's (commutative) functional calculus. 
We thus like to think of the present paper as establishing a posthumous connection between the works of these two friends.\\

{\em Notation and conventions.} All vector spaces are complex unless the contrary is stated. 
To be in line with standard notation in pseudo-differential calculus, we reserve the notation $(x,\xi)$ for the general point in $\R^{2d} = \R^d\times\R^d$. 
Because most applications are concerned with function spaces anyway, general elements in a Banach space $X$ will be denoted by $f,g,\dots$.
For $\xi\in \R^d$ we write $\lb \xi\rb = (1+|\xi|^2)^{1/2}$. 
Standard multi-index notation is used. We let $\N = \{0,1,2,\dots\}$.

When $A=(A_1,\dots,A_d)$ and $B = (B_1,\dots,B_d)$ 
are $d$-tuples of linear operators with domains $\Dom(A_j)$ and $\Dom(B_j)$ respectively, we set $\D(A) = \bigcap_{j=1}^d \D(A_j)$ and $\D(B)=\bigcap_{j=1}^d \D(B_j).$
For $u,v\in \R^d$ we write
$uv:= \sum_{j=1}^d u_jv_j$ and define the operators $uA$ and $vB$,  with domains $\Dom(A)$ and $\Dom(B)$ respectively, by
$$uA= \sum_{j=1}^d u_jA_j, \quad vB= \sum_{j=1}^d v_jB_j. $$

We write $a\lesssim_{p_1,p_2,\dots} b$ to express that there exists a constant $C$, 
depending on the data $p_1,p_2,\dots$, but not on any other relevant data,
such that $a\le Cb$. If the  constant is independent of all relevant data we write $a\lesssim b$,

\section{Preliminaries}

We assume familiarity with the basic theory of pseudo-differential operators and semigroup theory. Good sources for our purposes are \cite{abels, Stein} and \cite{EngNag}.
Here we collect some terminology and results concerning UMD spaces, 
$R$-boundedness, and the $H^\infty$-calculus of sectorial operators. Our main references are \cite{HNVW1, HNVW2}; 
other sources for these notions are, respectively, \cite{Pisier}, \cite{DHP, KunWei}, \cite{Haase, Haase-ISEM, KunWei}. 

\subsection{UMD spaces}

A Banach space $X$ is said to have the {\em UMD$_p$ property}, where $1<p<\infty$, if there exists a finite constant $C\ge 0$ such that whenever $(m_n)_{n=1}^N$ is a finite $X$-valued martingale (defined on a measure space which may vary from case to case and whose length $N$ may vary as well) and $(\epsilon_n)_{n=1}^N$ is a sequence of scalars of modulus one, we have 
$$ \E \Big\n \sumnN \epsilon_n m_n \Big\n^p \le C^p  \E \Big\n \sumnN m_n \Big\n^p.$$ 
It can be shown that if $X$ has the UMD$_p$ property for some $1<p<\infty$, then it has this property for all $1<p<\infty$. Accordingly it makes sense to call a Banach space a {\em UMD space} if it has the UMD$_p$ property for some (equivalently, for all) $1<p<\infty$.

In some treatments only scalars $\epsilon_n = \pm 1$ are used. This leads to an equivalent 
definition, the only difference being that the numerical value of the constant may change (see \cite[Proposition 4.2.10]{HNVW1}). 

The importance of the class of UMD spaces derives from a celebrated theorem due Burk\-holder and Bourgain 
\cite{Bour83, Burk83} which characterises it as precisely the class of Banach spaces $X$ for which the Hilbert transform extends to a bounded operator on $L^p(\R;X)$ for some (equivalently, for all) $1<p<\infty$. This, in turn, allows one to prove the boundedness in $L^p(\R^d;X)$ 
of very general classes of singular integral operators. For some of the sharpest results presently available see \cite{Hyt-vvTb}. In particular 
every Calder\'on--Zygmund operator with a kernel satisfying the so-called ``standard estimates'' is bounded on $L^p(\R^d;X)$ for all UMD spaces $X$ and exponents $1 < p < \infty$.

Examples of UMD spaces include Hilbert spaces, the $L^p$-spaces with $1<p<\infty$, and the Schatten classes $\mathscr{C}_p$ with $1<p<\infty$. The class of UMD spaces is stable under passing to equivalent norms and taking closed subspaces, quotients, and $\ell^p$-direct sums. If $X$ is UMD and $1<p<\infty$, then also $L^p(M,\mu;X)$ 
is UMD, for any measure space $(M,\mu)$. As a consequence, all ``classical'' function spaces used in Analysis such as Sobolev spaces, Besov spaces, and Triebel--Lizorkin spaces are UMD as long as the exponents in their definitions are within the reflexive range. UMD spaces are reflexive, and therefore spaces such as $c_0$, $\ell^1$, $\ell^\infty$, $C(K)$, $L^1(M,\mu)$, $L^\infty(M,\mu)$
are not UMD (with exception of the trivial cases when the latter three are finite-dimensional).

\subsection{$R$-boundedness}
A {\em Rademacher sequence} is a sequence of independent random variables $(\eps_n)_{n=1}^\infty$, defined on some probability space, the values of which are uniformly distributed in the set of scalars of modulus one. Thus if the scalar field is real, Rademacher variables take values $\pm 1$ with
with equal probability $\frac12$, and if the scalar field is complex their values are
uniformly distributed in the unit circle in the complex plane.

Let $X$ and $Y$ be Banach spaces and let $\calL(X,Y)$ denote the space of all bounded linear operators from $X$ into $Y$. A subset $\mathscr{T}$ of $\calL(X,Y)$ is said to be {\em $R_p$-bounded}, where $0<p<\infty$, if there exists a finite constant $C\ge 0$ such that for all finite sequences
$T_1,\dots,T_N\in \mathscr{T}$ and $x_1,\dots,x_N\in X$ (where $N$ may vary) one has
$$ \E \Big\n \sumnN \eps_n T_n x_n\Big\n^p \le C^p \E \Big\n \sumnN \eps_n x_n\Big\n^p.$$ 
The least admissible constant $C$ is called the {\em $R_p$-bound} of $\mathscr{T}$ and is denoted by $\mathscr{R}_p(\mathscr{T})$. 

By the Kahane--Khintchine inequality (see \cite[Theorem 6.2.4]{HNVW2}), if $\mathscr{T}$ is $R_p$-bounded for some
$0<p<\infty$, then it is $R_p$-bounded for all $0<p<\infty$, and for all $0<p<\infty$ we have
$$ \mathscr{R}_p(\mathscr{T})\eqsim_{p} \mathscr{R}(\mathscr{T}),$$
where by default we write $\mathscr{R}(\mathscr{T}):= \mathscr{R}_2(\mathscr{T})$.
Accordingly it makes sense to call $\mathscr{T}$ {\em $R$-bounded} if it is $R$-bounded for some (equivalently, for all) $0<p<\infty$. 

In some treatments real-valued Rademacher variables (random variables taking the values $\pm 1$ with equal probability) are used. This leads to an equivalent 
definition, the only difference being that the numerical value of the $R$-bounds 
may change (see \cite[Proposition 6.1.9]{HNVW2}). Upon replacing the role of Rademacher variables by Gaussian variables,
one arrives at the notion of {\em $\gamma$-boundedness}. 
Every $R$-bounded set of operators is $\gamma$-bounded (by a simple randomisation argument, see \cite[Theorem 8.1.3]{HNVW2}), and every $\gamma$-bounded set is uniformly bounded (take $N=1$). If $X$ has finite cotype,
every $\gamma$-bounded family in $\calL(X,Y)$ is $R$-bounded, 
and if $X$ has cotype $2$ and $Y$ has type $2$ (in particular, if $X$ and $Y$ are isomorphic to Hilbert spaces),
then every uniformly bounded family in $\calL(X,Y)$ is $R$-bounded (see \cite[Theorem 8.1.3]{HNVW2}).
The Kahane contraction principle (see \cite[Theorem 6.1.13]{HNVW2}) implies that 
bounded subsets of the scalar field, viewed as bounded operators on a Banach space $X$ though scalar multiplication, are $R$-bounded. $R$-Bounded sets enjoy many permanence properties; in particular they are closed under taking convex hulls and weak operator closure (see \cite[Sections 8.1.e, 8.4.a, 8.5.a]{HNVW2}).

The notion of $R$-boundedness originates from Harmonic Analysis, where it captures the essence of so-called ``square function estimates''. As such it goes back to the works 
\cite{BG, Bour86}; its first systematic study is \cite{CPSW}.
Rather than explaining this aspect in full detail (for this we refer to \cite[Chapter 8]{HNVW2}) we mention (see \cite[Proposition 6.3.3]{HNVW2}) that if $X = L^q(M,\mu)$ with $1\le q<\infty$, then for all $0<p<\infty$
one has the equivalence of norms 
$$ \Big(\E \Big\n \sumnN \eps_n f_n\Big\n_{L^q(M,\mu)}^p\Big)^{1/p}
\eqsim_{p,q} \Big\n \Big( \sumnN |f_n|^2\Big)^{1/2}\Big\n_{L^q(M,\mu)}
$$
with implied constants that depend only on $p$ and $q$. Thus, in the context of $L^q$-spaces,
$R$-boundedness reduces to a square function estimate.

\subsection{$H^\infty$-calculus}\label{subsec:Hinfty}

Let $X$ be a Banach space and let $0<\sigma<\pi$. 
A closed operator $L: \Dom(L)\supseteq X\to X$ (with $\Dom(L)$ the {\em domain} of $L$)
is said to be {\em $\sigma$-sectorial} if its spectrum is contained in the closure of the sector
$$\Sigma_\sigma = \{z\in \C:\ z\not=0, \ |\arg z| < \sigma\}$$
(arguments are taken in $(-\pi,-\pi)$) and satisfies
\begin{align}\label{eq:sectorial}
\n R(z,L)\n \le \frac{M}{|z|}
\end{align}
on the complement of $\overline{\Sigma_\sigma}$, for some finite constant $M\ge 0$. Here, $R(z,L) := (z-L)^{-1}$
is the resolvent operator. An operator is said to be {\em sectorial} if it is $\sigma$-sectorial for some $0<\sigma<\pi$. The number
$$\om(L):= \inf\bigl\{\sigma\in (0,\pi): \ \hbox{$L$ is $\sigma$-sectorial}\big\}$$
is called the {\em angle of sectoriality of $L$}.

For $0<\theta<\pi$ let $H^1(\Sigma_\theta)$ be the Banach space of all holomorphic functions $\phi:\Sigma_\theta\to \C$ satisfying
$$ \n \phi\n_{H^1(\Sigma_\theta)}:= \sup_{0<\nu<\theta} \frac1{2\pi} \int_{\partial \Sigma_\nu} |\phi(z)|\,\frac{|{\rm d}z|}{|z|} <\infty.$$
If $L$ is $\sigma$-sectorial, then for any $\phi\in H^1(\Sigma_\theta)$ with $\sigma<\theta<\pi$ we may define 
\begin{equation}\label{eq:Dunford} \phi(L) :=  \frac1{2\pi i} \int_{\partial \Sigma_\nu} \phi(z) R(z,L)\, {\rm d}z,
\end{equation}
taking $\sigma<\nu<\theta$ with the understanding that $\partial \Sigma_\nu$ is downwards oriented. This integral converges absolutely and defines 
a bounded operator of norm at most $M\n \phi\n_{H^1(\Sigma_\theta)}$, where $M$ is the constant of \eqref{eq:sectorial}. It is a consequence of Cauchy's theoremand \cite[Proposition H.2.5]{HNVW2} that 
the definition of $\phi(L)$ is independent of the choice of the angle $\nu$.

If we were to replace the role of $H^1(\Sigma_\theta)$ by 
the space $H^\infty(\Sigma_\theta)$ of all bounded holomorphic functions on $\Sigma_\theta$, 
we would run into the difficulty that the corresponding Dunford integral in \eqref{eq:Dunford} becomes singular at both the origin and at infinity. To handle this situation\, a sectorial
operator $L$ is said to have a {\em bounded $H^\infty(\Sigma_\sigma)$-calculus}, where $\omega(L)<\sigma<\pi$, if there exists a finite constant $K\ge 0$
such that 
$$ \n \phi(L)\n \le K \n \phi\n_{H^\infty(\Sigma_\sigma)}$$
for all $\phi\in H^1(\Sigma_\sigma)\cap H^\infty(\Sigma_\sigma)$.
A sectorial operator $L$ is said to have a {\em bounded $H^\infty$-calculus} if it has a bounded $H^\infty(\Sigma_\sigma)$-calculus for some $\omega(L)<\sigma<\pi$.
The number
$$\om_{H^\infty}(L):= \inf\bigl\{\sigma\in (\omega(L),\pi): \ \hbox{$L$ has a bounded $H^\infty(\Sigma_\sigma)$-calculus}\bigr\}
$$
is called the {\em angle of the $H^\infty$-calculus} of $L$.

If $L$ is densely defined, has dense range, and has a bounded $H^\infty(\Sigma_\sigma)$-calculus, the {\em McIntosh convergence lemma} \cite{McI} (see also \cite[Theorem 10.2.13]{HNVW2}) allows one to uniquely define, for every $\phi\in H^\infty(\Sigma_\sigma)$, a bounded operator $\phi(L)$ by
$$\phi(L)f:= \limn \phi_n(L)f, \quad f\in X,$$
where $(\phi_n)_{n\ge 1}$ is any sequence in $H^1(\Sigma_\sigma)\cap H^\infty(\Sigma_\sigma)$
that is uniformly bounded and converges to $\phi$ pointwise on $\Sigma_\sigma$.

The prime example of a sectorial operator with a bounded $H^\infty$-calculus (of angle $0$)
is the negative Laplacian $L = -\Delta$ on $L^p(\R^d;X)$ for any UMD space $X$ and $1<p<\infty$. More generally, under minor regularity assumptions on the coefficients, uniformly elliptic operators on sufficiently regular domains $D$ in $\R^d$ have bounded $H^\infty$-calculi on $L^p(D;X)$ under various boundary conditions. Examples with precise formulations are reviewed in \cite{DHP, HNVW2, KunWei}.

There is an interesting interplay between $R$-boundedness and $H^\infty$-calculi.
Let us say that a closed operator $L$ is {\em $\sigma$-$R$-sectorial}
if $\sigma(L)$ is contained in $\overline{\Sigma_\sigma}$ and
the set $$\{z R(z,L): \ z\in \complement\overline{\Sigma_\sigma}\big\}$$ is $R$-bounded.
Since $R$-boundedness implies boundedness, every $\sigma$-$R$-sectorial is $\sigma$-sectorial.
The operator $L$ is said to be {\em $R$-sectorial} if it is $\sigma$-$R$-sectorial
for some $0<\sigma<\pi$. The infimum
$$ \omega_R(L) := \inf\big\{\sigma\in (\omega(L),\pi): \ \hbox{$L$ is $\sigma$-$R$-sectorial}\big\}$$
is called the {\em angle of $R$-sectoriality} of $L$.
It was shown by Kalton and Weis \cite{KalWei} (see also \cite[Corollary 10.4.10]{HNVW2}) that if $L$ is a sectorial operator with a bounded $H^\infty$-calculus
on a UMD Banach space $X$ (actually a slightly weaker assumption will do for this purpose, but this is not relevant to us here), then $L$ is $R$-sectorial and we have 
\begin{equation}\label{eq:KW-Hinfty} \om_{R}(L) = \om_{H^\infty}(L).
\end{equation}
In this context it is interesting to observe that for $R$-sectorial operators $L$ it may happen that $\omega_R(L)> \omega(L)$; see \cite{KLW}.

\section{Weyl pairs}\label{sec:Weyl}

Let $A=(A_1,\dots,A_d)$ and $B = (B_1,\dots,B_d)$ 
be two $d$-tuples of closed and densely defined operators acting in a complex
Banach space $X$. We assume that each of the operators $iA_j$ and $iB_j$ generates a uniformly 
bounded $C_0$-group on $X$. We denote these groups by  $(e^{itA_j})_{t\in \R}$ 
and $(e^{itB_j})_{t\in \R}$, respectively. 

\begin{definition}
Under the above assumptions, the pair $(A,B)$ will be called a {\em Weyl pair} of dimension $d$ if the (integrated) {\em canonical commutation relations} hold for all $s,t\in \R$ and $1\le j,k\le d$:
\begin{equation}\label{eq:CCR}
\begin{aligned}
e^{isA_j}e^{itA_k} &= e^{itA_k}e^{isA_j}\\
e^{isB_j}e^{itB_k} &= e^{itB_k}e^{isB_j}\\
      e^{isA_j}e^{itB_k} & = e^{-ist \delta_{jk}} e^{itB_k}e^{isA_j} 
\end{aligned}
\end{equation}
where $\delta_{jk}$ is the usual Kronecker symbol.
\end{definition}

Being a Weyl pair is an isomorphic notion, in that it is insensitive to changing to an equivalent norm. More generally, if $(A,B)$ is a Weyl pair on $X$ and $T:X\to Y$ is an isomorphism of Banach spaces,  then $(TAT^{-1}, TBT^{-1})$ is a Weyl pair on $Y$. 
This is of course trivial, but it is of some interest in connection with the next example, for on Hilbert spaces it easily provides examples of non-selfadjoint Weyl pairs. 
 
\begin{example}[Standard position/momentum pair]\label{ex:HO} On $L^p(\R^d)$, $1\le p<\infty$, the position and momentum operators 
$Q_j$ and $P_j$, $1\le j\le d$, are defined by 
\begin{align*}
Q_jf(x) = x_jf(x), \quad P_jf(x) = \frac1i \partial_j f(x), \quad x\in \R^d.
\end{align*}
With their natural domains, it is easily checked that they define a Weyl pair
$(Q,P)$. Indeed, $iQ_j$ generates the multiplication group on $L^p(\R^d)$ given by $$e^{itQ_j}g(x) =  e^{itx}g(x), \quad x\in \R^d, \ t\in \R,$$
and
$iP_j$ generates the translation group on $L^{p}(\R^d)$ given by
$$e^{itP_j}g(x) = g(x+te_j), \quad x\in\R^d, \ t\in \R,$$ with $e_j$ the $j$-th unit vector of $\R^d$.
The commutation relations are easily checked.

The position/momentum pair is sometimes referred to as the {\em standard pair} and provides the main example of a Weyl pair. A well-known uniqueness result of Stone and von Neumann (see, e.g., \cite[Chapter 14]{Hall} or \cite[Section 4.3]{Put}) asserts that every Weyl pair of dimension $d$ of self-adjoint operators in a Hilbert space is unitarily equivalent to a direct sum of copies of standard pairs on $L^2(\R^d)$. 
\end{example}

\begin{example}[Gaussian position/momentum pair]\label{ex:OU} 
Let us denote by $\gamma$ the standard Gaussian measure
on $\R^d$. On $L^p(\R^d,\gamma)$, $1\le p<\infty$, we consider 
the position and momentum pair $(Q^\gamma,P^\gamma)$ given by $Q^\gamma = (Q^\gamma_1,\dots,Q^\gamma_d)$ and $P^\gamma = (P^\gamma_1,\dots,P^\gamma_d)$
defined by 
$$ Q^\gamma_j :=\frac1{\sqrt2}(a_j + a^\dagger_j), \quad P^\gamma_j :=\frac1{i\sqrt{2}}(a_j - a_j^\dagger),$$
where the annihilation and creation operators $a_j$ and $a_j^\dagger$ are defined by 
$$ a_j = \partial_j, \quad a_j^\dagger = -\partial_j+x_j.$$
Thus, for $f\in C_{\rm c}^1(\R^d)$,
$$Q^\gamma_jf(x) = \frac1{\sqrt 2} x_j f(x), \quad P^\gamma_jf(x) = \frac1{i\sqrt 2}(2\partial_j - x_j) f(x).$$
It is readily verified that the pair $(Q^\gamma,P^\gamma)$ satisfies the canonical commutation relations.
As we will explain in a moment, for $p=2$ this pair is unitarily equivalent to the standard pair.

It is clear that the operators $iQ^\gamma_{j}$ generate 
$C_0$-contraction groups of multiplication operators on $L^p(\R^d,\gamma)$ for all $1\le p<\infty$. 
On the other hand, the operators $iP^\gamma_{j}$ generate bounded $C_0$-groups on $L^p(\R^d,\gamma)$ if and only if $p=2$. 
Thus $(Q^\gamma,P^\gamma)$ is a Weyl pair on  $L^p(\R^d,\gamma)$ if and only if $p=2$.
This can be deduced from Theorem \ref{thm:A2B2} below as follows. By a result of \cite{NP}, in $L^2(\R^d,\gamma)$ the operator $\frac12((Q^\gamma)^2+(P^\gamma)^2)-\frac12d$ considered in Theorem \ref{thm:A2B2} 
is the Ornstein--Uhlenbeck operator. If $(Q^\gamma,P^\gamma)$ were to be a Weyl pair in 
$L^p(\R^d,\gamma)$ for certain $p\in (1,\infty)\setminus\{2\}$, the theorem would imply that the Ornstein--Uhlenbeck semigroup extends 
holomorphically to the right half-plane $\{\Re z>0\}$, and this is well known to be false.
In fact the optimal angle $\theta_p$ of holomorphy for the Ornstein--Uhlenbeck semigroup 
on $L^p(\R^d,\gamma)$ is known to be
$\cos\theta_p = \frac{|p-2|}{2\sqrt{p-1}}$ (see \cite{GMMST}).

The failure of $iP^\gamma_j$ to generate a bounded $C_0$-group on $L^p(\R^d,\gamma)$ for $p\not=2$ can also be easily checked by hand.
Let $m(\dd x) = \dd x/(2\pi)^{d/2}$ denote the normalised Lebesgue measure on $\R^d$. 
On $L^2(\R^d,\gamma)$ the group generated by $iP^\gamma_j$ is given by 
$e^{itP^\gamma_j} = U^{-1}T_j(t)U$, where $T_j(t)$ is the translation group on $L^2(\R^d,m)$
in the $j$-th direction
and $U : L^2(\R^d,\gamma) \to L^2(\R^d,m)$ is the unitary mapping given by
$U = \delta\circ E$ with $$ Ef(x) = e^{-\frac14|x|^2}f(x), \quad \delta f(x) := (\sqrt 2)^d f\bigl({\sqrt 2}x\bigr).$$
An easy computation shows that, in $L^2(\R^d,\gamma)$, the operators  $e^{itP^\gamma_j}$ are given by 
\begin{align*}
  e^{itP^\gamma_j}f(x) = e^{\frac14|x|^2 - \frac12(\frac{x}{\sqrt 2}-t)^2}f(x+t\sqrt 2). 
\end{align*}
Then, after an integration and change of variable,
$$ \n e^{itP^\gamma_j}f\n_p^p = \frac{1}{(2\pi)^{d/2}}\int_{\R^d} e^{(\frac12-\frac{p}{4})(2{\sqrt 2}xt -2t^2)}|f(x)|^p\ud \gamma(x).$$
For $p\in [1,2)$ it follows that $e^{itP^\gamma_j}$ fails to extend to a bounded operator on $L^p(\R^d,\gamma)$ for all $t>0$,
and for $p\in (2,\infty)$ the operators $e^{itP^\gamma_j}$ are bounded on $L^p(\R^d,\gamma)$, but not uniformly bounded as a function of $t>0$. 
\end{example}

\begin{example}[Modified Gaussian position/momentum pair]
It is of some interest to note that the pair $(Q^\gamma,P^\gamma)$ of the previous example does form a Weyl pair on $L^p(\R^d,\gamma_{2/p})$ for all $p\in [1,2]$, where
$\gamma_\tau(\dd x) = (2\pi\tau)^{-d/2}e^{-|x|^2/2\tau}\ud x$. 
This is simply because with this scaling of the measure the mapping $U$ considered above defines an isomorphism from $L^p(\R^d,\gamma_{2/p})$ onto $L^p(\R^d,m)$.
Then each $iP^\gamma_j$ generates a bounded $C_0$-group on $L^p(\R^d,\gamma_{2/p})$ which, 
under $U$, conjugates with the translation group in the $j$-th direction on $L^p(\R^d,m)$.
\end{example}

\begin{example}[Duality]
If $(A,B)$ is a Weyl pair in $X$, then the pair of adjoint operators $(B^*,A^*)$ 
is a Weyl pair in $X^{*}$ provided the operators $A_j^*$ and $B_j^*$ are densely defined (by a classical result in semigroup theory (see \cite[Proposition I.5.14]{EngNag}) this is always the case if $X$ is reflexive).
\end{example}

\begin{example}[Additive commuting perturbations]
 If $(A,B)$ is a Weyl pair and 
 $C$ is a bounded operator resolvent commuting with $A$, then 
 $(A,B+C)$ is a Weyl pair whenever the group generated by $i(B+C)$ is bounded. Indeed, the assumption implies that $C$ commutes with the operators $e^{itA}$, and the commutation relations \eqref{eq:CCR} follow from this by going through the standard proof of the variation of constants formula for perturbed (semi)groups using Picard iteration.
 The simplest example is obtained by taking $C = \omega I$ with $\omega\in\R$. This amounts to frequency modulating the group generated by $iB$. More generally one could take $C$ to 
 be any densely defined closed operator such that $iC$ generates a bounded group commuting with the group generated by $iA$. 
 
 Similarly,  if $(A,B)$ is a Weyl pair and $C$ is a bounded operator commuting with the resolvent of $B$, then $(A+C,B)$ is a Weyl pair whenever the group generated by $i(A+C)$ is bounded.
\end{example}

\begin{example}[Skew transforms] If $(A,B)$ is a Weyl pair, then for every $\lambda\in \R$ the pair
 $(A,\lambda A+B)$ is a Weyl pair. Some care has to be taken with the interpretation of 
 $\lambda A+B$; we interpret it as the generator of the $C_0$-group given by
 $$ e^{it(\lambda A+B)}:= e^{\frac12 i\lambda t^2}e^{i\lambda t A}e^{it B} $$
(this idea will be further developed in a moment). Similarly, if $(A,B)$ is a Weyl pair, then for every $\lambda \in \R$ the pair
 $(A+\lambda B,B)$ is a Weyl pair. 
\end{example}

\begin{example}
Let $((Q_1,Q_2), (P_1,P_2))$ be the standard pair of dimension $2d$ on $L^2(\R^{2d}),$ i.e.,
\begin{align*}
Q_{1,j}f(x,\xi) = x_jf(x,\xi), \quad & Q_{2,j}f(x,\xi) = \xi_jf(x,\xi), \\ 
P_{1,j}f(x,\xi) = \frac1i \frac{\partial f}{\partial x_j}(x,\xi),\quad   &P_{2,j}f(x,\xi) = \frac1i \frac{\partial f}{\partial \xi_j}(x,\xi),
\end{align*}
for $1\le j\le d$. Reasoning as in the preceding examples, we see that $(-\frac12 Q_2-P_1,\frac12Q_1-P_2)$ is a Weyl pair of dimension $d$ on $L^2(\R^{2d})$. As we show in Lemma \ref{lem:twiststand}, the Weyl calculus of this pair encodes twisted convolutions. 
Many variations on twisted convolutions can be considered through the Weyl calculus of twisted standard pairs obtained from different twists than the one above.

\end{example}

\begin{example}[Quantum variables]
In \cite{qes}, Gonz\'alez-P\'erez, Junge, and Parcet, introduce a (non-commutative) Fourier transform, as well as position and momentum operators,
associated with certain von Neumann algebras called quantum euclidean spaces (or Moyal deformations, or CCR algebras). Their construction 
allows 
them to define non-commutative analogues of the key notions of Calder\'on-Zygmund theory, including off-diagonal kernel estimates and H\"ormander symbol classes, and then to prove analogues of the main theorems in singular integral operator theory. We cannot describe their construction in detail here, but note that their quantum variables $(x_{\Theta,j})_{j=1,...,2d}$ are Weyl pairs (for the appropriate choice of $\Theta$) acting on some non-commutative $L^p$-spaces (see \cite[Proposition 1.9]{qes}). 
\end{example}

We now collect some easy properties of Weyl pairs which will be useful later on. 
For $d=1$ they are due to Kato \cite{Kato} (see also \cite[Section 4.9]{Put})
and the proofs given there extend without difficulty to the present case. The main
observation is that, upon taking Laplace transforms, the third commutation relation in \eqref{eq:CCR} implies
the identities
\begin{equation}\label{eq:CCR-res}
\begin{aligned} R(\lambda,iA_j)e^{itB_j} & = e^{itB_j}R(\lambda+it,iA_j), \\
R(\lambda,iB_j)  e^{itA_j} &= e^{itA_j}R(\lambda-it,iB_j),
\end{aligned}
\end{equation}
for all $t\in \R$, $\Re\lambda\not=0$, and $1\le j\le d$.
It follows that $e^{itB_j}$ leaves $\Dom(A_j)$ invariant,
$e^{itA_j}$ leaves $\Dom(B_j)$ invariant, and 
\begin{equation}\label{eq:CCR-op1}
\begin{aligned}
A_je^{itB_j}f & = e^{itB_j}(A_j-t)f , \quad f\in \Dom(A_j),\\
 B_j e^{itA_j}f & =e^{itA_j}(B_j+t)f, \quad f\in \Dom(B_j).
\end{aligned}
\end{equation}
The same argument applies to the remaining combinations of $A_j$ and $B_k$, but no shifts over $\pm t$ occur when the operators commute. Thus we obtain:

\begin{lemma}\label{lem:Kato} 
Let $(A,B)$ be a Weyl pair. The operators $e^{itA_j}$ and $e^{itB_j}$ leave both $\Dom(A):= \bigcap_{k=1}^d \Dom(A_k)$ and $\Dom(B):= \bigcap_{k=1}^d \Dom(B_k)$ invariant. For $j\not=k$ we have 
\begin{equation}\label{eq:CCR-op2}
\begin{aligned}
 A_je^{itB_k}f  = e^{itB_k}A_jf, \quad f\in \Dom(A_j),\\
 B_k e^{itA_j}f  =e^{itA_j}B_kf, \quad f\in \Dom(B_k),
\end{aligned}
\end{equation}
while for $j=k$ the identities \eqref{eq:CCR-op1} hold. 
\end{lemma}

Differentiating \eqref{eq:CCR-res} at $t=0$ gives 
\begin{equation}\label{eq:kato}
\begin{aligned}
R(\lambda,iB_j)R(\mu,iA_j) & = R(\mu,iA_j) R(\lambda,iB_j)[I -iR(\lambda,iB_j)R(\mu,iA_j)], \\ 
R(\lambda,iA_j)R(\mu,iB_j) &= R(\mu,iB_j) R(\lambda,iA_j)[I +i R(\lambda,iA_j)R(\mu,iB_j)].
\end{aligned}
\end{equation}
If 
$$g = \prod_{j,j'=1}^d R(\lambda_j,iA_j)R(\lambda_{j'},iA_{j'}) \prod_{k,k'=1}^d R(\mu_k,iB_k) R(\mu_{k'},iB_{k'})f$$ with $f\in X$, then 
 \eqref{eq:kato} may be used to rewrite $g$, for any pair $1\le j,k\le d$, as
\begin{align*} g & =  R(\lambda,iA_j)R(\mu,iB_k) C_{jk}f \\ & = R(\mu,iB_k) R(\lambda,iA_j)D_{jk}f 
\\ & = R(\mu,iB_j) R(\lambda,iB_k)E_{jk}f\end{align*}
for suitable bounded operators $C_{jk}$, $D_{jk}$, $E_{jk}$.
From this we  see that  $g$ belongs to $\bigcap_{1\le j,k\le d}(\Dom(A_jA_k)\cap\Dom(A_jB_k)\cap\Dom(B_kA_j))\cap \Dom(B_jB_k)$. Since
$$\lim_{\lambda,\mu\to\infty} \prod_{j,j',k,k'=1}^d \lambda_j \lambda_{j'}R(\lambda_j,iA_j)R(\lambda_{j'},iA_{j'})\mu_k\mu_{k'} R(\mu_k,iB_k)R(\lambda_{k'},iB_{k'})f = f$$
 for all $f\in X$,
 the limit being taken in any order for $\lambda_1,\dots,\lambda_d$, $\lambda_1',\dots,\lambda_d'$, $\mu_1,\dots,\mu_d$, $\mu_1,'\dots,\mu_d'\to \infty$,
 this subspace is dense in $X$. 
The identity \eqref{eq:kato} also gives the identity
$A_jB_j g - B_jA_j g = ig$ for $g$ of the above form. 
The same argument gives commutation for the remaining combinations of $A_j$ and $B_k$. Thus we obtain:

\begin{lemma}\label{lem:D} Let $(A,B)$ be a Weyl pair of dimension $d$.
The subspace 
$$ \bigcap_{1\le j,k\le d}(\Dom(A_jA_k)\cap \Dom(A_jB_k)\cap \Dom(B_kA_j)\cap\Dom(B_jB_k))$$ 
is dense, 
and on this subspace we have $A_jA_k = A_kA_j$,  $B_jB_k = B_kB_j$,
and $A_jB_k - B_kA_j = \delta_{jk} iI.$ 
\end{lemma}

Let $(A,B)$ be a Weyl pair of dimension $d$. Consider, for $t\in \R$ and $u,v\in \R^d$, the bounded operators
$$T_{u,v}(t):=  e^{\frac12 it^2uv}e^{itu A}e^{itv B} = e^{-\frac12 it^2uv} e^{itv B}e^{itu A}.$$

\begin{proposition} \label{prop:core}
 The family $(T_{u,v})_{t\in\R}$ is a bounded $C_0$-group on $X$, 
$\D(A)\cap \D(B)$ is a core for its generator $G_{u,v}$, and, on this core,
the generator is given by 
$$ G_{u,v}f  = iu A f+iv Bf, \quad f\in\D(A)\cap \D(B).$$
\end{proposition}
\begin{proof}
The identity $T_{u,v}(0)=I$ is trivial.
The group property $T_{u,v}(t_0)\circ T_{u,v}(t_1) = T_{u,v}(t_0+t_1)$ follows straightforwardly from the commutation relations \eqref{eq:CCR}.
Strong continuity is also clear.

It follows from the general properties of Weyl pairs mentioned earlier that 
each operator $T_{u,v}(t)$ maps the subspace $\D(A)\cap \D(B)$ into itself. 
Moreover, $\D(A)\cap\D(B)$ is dense in $X$. By Lemma \ref{lem:diff} below,  every $f\in \D(A)\cap\D(B)$ belongs to $\D(G_{u,v})$ and 
differentiation gives $$ G_{u,v}f  = \frac{{\rm d}}{{\rm d}t}\Big|_{t= 0}T_{u,v}(t)f =  iu Af +iv Bf,
\quad f\in \Dom(A)\cap\Dom(B).$$
A general result in semigroup theory (see, e.g., \cite[Proposition II.1.7]{EngNag}) now implies that
$\D(A)\cap \D(B)$ is a core for $G_{u,v}$.
\end{proof}
The proof of Proposition \ref{prop:core} is completed by the following observation, 
which we leave as an easy exercise to the reader.

\begin{lemma}\label{lem:diff}
 Let $(S(t))_{t\in\R}$ and $(T(t))_{t\in\R}$ be strongly continuous families of operators, and let $f\in X$ be fixed.
 If
 \begin{enumerate}
  \item $t\mapsto S(t)f$ is differentiable at $t=0$, with derivative $S'(0)f:= \frac{\rm d}{{\rm d}t}\big|_{t=0}S(t)f$,\vskip2pt
  \item $t\mapsto T(t)S(0)f$ is differentiable at $t=0$, with derivative $T'(0)S(0)f:=\frac{\rm d}{{\rm d}t}\big|_{t=0}T(t)S(0)f$,\vskip2pt
 \end{enumerate}
then $t\mapsto T(t)S(t)f$ is differentiable at $t=0$, with derivative $$\frac{{\rm d}}{{\rm d}t}\Big|_{t=0} T(t)S(t)f = T'(0)S(0)f + T(0)S'(0)f.$$
\end{lemma}

\section{The Weyl calculus}

Let $(A,B)$ be a Weyl pair of dimension $d$ on a Banach space $X$.
For $(x,\xi)\in \R^{2d}$ we consider the bounded operators
\begin{align}\label{eq:abstr-Schr}e^{i(uA+v B)}:=  e^{\frac12 iuv}e^{iu A}e^{iv B}.
\end{align}
This notation is justified by Proposition \ref{prop:core}. 

\begin{example} For the standard pair $(Q,P)$ on $L^2(\R^d)$, \eqref{eq:abstr-Schr} 
reduces to the 
 {\em Schr\"od\-inger representation}: the operators $e^{i(uQ+v P)}$ are unitary 
 on $L^2(\R^d)$ and given by
 $$ e^{i(uQ+v P)}f(x) =  e^{\frac12 iuv + iu x}f(x+v).$$
\end{example}

\begin{definition}[Weyl calculus]
Let $(A,B)$ be a Weyl pair of dimension $d$. For functions $a \in \mathscr{S}(\R^{2d})$ we define
\begin{align*}
 a(A,B)f := \frac1{ (2\pi)^d }\int_{\R^{2d}} \wh a(u,v) e^{i(uA+vB)}f \ud u\ud v,   \quad f\in X,
\end{align*}
where 
\begin{align*} \wh a(u,v) =  \frac1{(2\pi)^d}\int_{\R^{2d}} a(x,\xi) 
e^{-i(xu+\xi v)}\ud x\ud \xi 
\end{align*} is the Fourier--Plancherel transform of $a$.
The mapping $a\mapsto a(A,B)$ from $\mathscr{S}(\R^{2d})$ to $\calL(X)$ is called the {\em Weyl calculus} of $(A,B)$.
\end{definition}

An easy computation based on the identity
\begin{align}\label{eq:sigma} e^{i(uA+v B)}\circ e^{i(u'A+v'B)}
=  e^{\frac12i(u'v-uv')} e^{i(u+u')A+(v+v')B},
\end{align}
which follows from the commutation relations \eqref{eq:CCR}, gives the 
following analogue of the multiplicativity property of the functional calculus of a 
single operator: for all $a,b\in \mathscr{S}(\R^{2d})$ we have 
$$ a(A,B)\circ b(A,B) = (a\,\# \,b)(A,B),$$
where  $a\# b$ is the {\em Moyal product} of $a$ and $b$, given by (see \cite[Section XII.3.3]{Stein})
\begin{align*} \ & (a\,\#\,b)(x,\xi) \\ &\qquad = \frac{1}{\pi^{2d}} \int_{\R^{2d}}\int_{\R^{2d}} a(x+u, \xi+u') b(x+v, \xi+v') e^{-2i(vu'-uv')}\ud u\ud u'\ud v\ud v'.
\end{align*}

\begin{definition}\label{def:typeAB} Let $N,m \in \N$.
A Weyl pair $(A,B)$ is said to admit a {\em bounded Weyl calculus of type $(-N,m)$} if, for all $a \in \mathcal{S}(\R^{2d})$, we have
$$
\|a(A,B)\|\lesssim \underset{|\alpha|,|\beta|\le m}{\max}\ \underset{(x,\xi) \in\R^{2d}}{\sup}\ \lb \xi\rb^{N+|\alpha|}| \partial_{\xi} ^{\alpha}\partial_{x}^{\beta} a(x,\xi)|,
$$
with a constant independent of $a$.
The pair $(A,B)$ is said to admit a {\em bounded Weyl calculus of type $-N$}
if it admits a bounded Weyl calculus of type $(-N,m)$ for some $m\in\N$.
\end{definition}

In Subsection \ref{sec:stand} we will prove that if $X$ is a UMD space and $1<p<\infty$, then the standard
pair $(Q,P)$ has a bounded Weyl calculus of type $0$ on $L^p(\R^d;X)$.

The convergence lemma for the Dunford calculus for sectorial operators (see, e.g., \cite[Theorem 10.2.2]{HNVW2}) has the following analogue for the Weyl calculus:

\begin{lemma}[Convergence lemma]\label{lem:McI}
Let $(a_{n})_{n \in \N}$ be a sequence of Schwartz functions defined on $\R^{2d}$ and let $N\in \N$. 
There exist $m=m(d,N) \in \N$ 
and $M=M(d,N) \in \N$, both depending only on $d$ and $N$, 
such that the following holds.
If $(A,B)$ is a Weyl pair with a bounded Weyl calculus of type $(-N-1,m)$, and if
\begin{enumerate}
\item[\rm(i)]
for all multi-indices $\gamma \in \N^{d}$ with $|\gamma|\le M$ we have $\limn\partial^{\gamma}a_{n}= 0$ uniformly on compact sets, 
\item[\rm(ii)]
$
\displaystyle\sup_{n \in \N} \|a_{n}(A,B)\|<\infty,
$
\end{enumerate}
then $\limn a_{n}(A,B)f= 0$ for all $f\in X$.
\end{lemma}

Admittedly the formulation of this lemma is a bit awkward; the point here is that we need $(A,B)$
to be of type $(-N-1,m)$ for all $m\ge m_0$, where $m_0$ may depend on $N$ and $d$.
The proof of the lemma is based on an asymptotic expansion representation for  Moyal products of Schwartz functions.

\begin{lemma}\label{lem:appmoy}
There exists a sequence $(c_{\alpha})_{\alpha \in \N^{2d}}$ of complex numbers such that, for all $a,b \in \mathcal{S}(\R^{2d})$ and all integers $M\in \N$, 
there exists a function $r_{a,b;M+1} \in \mathscr{S}(\R^{2d})$ such that
$$
a(A,B)b(A,B) = \sum_{\substack{\alpha\in \N^{2d} \\ |\alpha|_\infty\le M}} c_{\alpha}\partial^{\alpha}(ab)(A,B) +r_{a,b;M+1}(A,B)
$$
whenever $(A,B)$ is a Weyl pair. Moreover, there exists an $m \in \N$, depending only on $d$ and $M$, such that if $(A,B)$ has a bounded Weyl calculus of type $(-M-1,m)$, then
\begin{align*}
\ & \|r_{a,b;M+1}(A,B)\|
\\ & \qquad\lesssim
\underset{\substack{\alpha', \beta', \alpha'', \beta''\in \N^{d} \\ |\alpha'|,|\beta'|, |\alpha''|, |\beta''| \leq m}}{\max}\ \underset{x,\xi\in \R^{2d}}{\sup}
\lb \xi\rb ^{\min(|\alpha'|,|\alpha''|)} |\partial^{\alpha'} _{\xi} \partial^{\beta'} _{x} a(x,\xi)\partial^{\alpha''} _{\xi} \partial^{\beta''} _{x}b(x,\xi)|.
\end{align*}
\end{lemma}

\begin{proof}
Let $a,b \in \mathcal{S}(\R^{2d})$.
Recall that $a(A,B)b(A,B) = (a \# b)(A,B)$, where $a\#b$ is the Moyal product of $a$ and $b$. 
By \cite[Theorem 3.16]{abels}, for any $M\geq 0$
there exists a function $r_{a,b;M+1}\in \mathcal{S}(\R^{2d})$ such that
\begin{equation}\label{eq:expansion} a \# b(x,\xi)\, = \sum _{\substack{\alpha\in \N^{d} \\ |\alpha| \le M}} \frac1{\alpha!}\frac1{i^\alpha}\partial_\xi^\alpha a(x,\xi)\partial_x^\alpha b(x,\xi) + r_{a,b;M+1}(x,\xi).
\end{equation}
This gives the formula in the first part of the theorem (with many coefficients $c_\alpha$ equal to $0$).

Suppose next that $(A,B)$ has a bounded Weyl calculus of type $(-M-1,m)$ for some 
$M\in \N$, where $m\in \N$ is arbitrary for the moment but will be fixed later.
Then, by assumption,  the remainder $r_{a,b;M+1}(A,B)$ in 
the expansion \eqref{eq:expansion} for this particular value of $M$ satisfies the estimate
$$
\|r_{a,b;M+1}(A,B)\|\lesssim \underset{|\gamma|,|\delta|\le m}{\max}\ \underset{(x,\xi) \in\R^{2d}}{\sup}\ \lb \xi\rb^{M+1+|\gamma|}| \partial_{\xi} ^{\gamma}\partial_{x}^{\delta} r_{a,b;M+1}(x,\xi)|
$$
with a constant only depending on $M$, $m$ and the pair $(A,B)$. 
By \cite[Theorem 3.15]{abels},
$r_{a,b,M+1}(x,\xi)$ is given by a finite linear combination, extending over all multi-indices
satisfying $|\alpha|=M+1$,  
of terms of the form
$$
R_{\alpha,a,b}(x,\xi) := \int_{\R^{2d}}e^{-ix'\xi'}(\xi')^\alpha \int _{0} ^{1} \partial^{\alpha}_{\xi} p(x,\xi+\theta \xi',x+x',\xi)(1-\theta)^{M}d\theta\ud x'\ud \xi'
$$
for $p(x,\xi,x',\xi')=a(x,\xi)b(x',\xi').$

As in the proof of  \cite[Theorem 3.15]{abels} (see, in particular, (3.20) on page 54 and (3.10) on page 47), there exists $m(d,M) \in \N$, depending only on $d$ and $M$, such that for all multi-indices $\gamma,\delta$ satisfying $|\gamma|, |\delta|\le m(d,M)$ we have
$$
|\partial^{\gamma}  _{\xi} \partial^{\delta} _{x} R_{\alpha,a,b}(x,\xi)|
\lesssim \lb \xi\rb ^{-(|\alpha|+|\gamma|)} = \lb \xi\rb ^{-(M+1+|\gamma|)},$$
with constant depending linearly on 
$$
\underset{|\alpha'|,|\beta'|, |\alpha''|, |\beta''| \leq m(d,M)}{\max}\ \underset{x,\xi\in \R^{2d}}{\sup}\
\lb \xi\rb ^{\min(|\alpha'|,|\alpha''|)} |\partial^{\alpha'} _{\xi} \partial^{\beta'} _{x} a(x,\xi)\partial^{\alpha''} _{\xi} \partial^{\beta''} _{x}b(x,\xi)|.
$$
If we fix the integer $m$ to be this $m(d,M)$, the second part of the lemma 
follows by collecting estimates.
\end{proof}

The proof of the convergence lemma requires one further auxiliary result.
Given a function $\eta:\R^{2d}\to \C$ and a real number $\delta>0$ we set
$\eta_{\delta}(x,\xi):= \eta(\delta x,\delta\xi)$.

\begin{lemma}\label{lem:appid}
For all $\eta \in C^{\infty}_{\rm c}(\R^{2d})$ with $\eta(0,0)=1$,
and $f\in X$, we have $$\limk\eta_{\frac{1}{k}}(A,B)f = f.$$
\end{lemma}

\begin{proof}
For all $f \in X$ we have 
\begin{align*}
\eta_{\frac{1}{k}}(A,B)f 
& = \frac1{(2\pi)^d}\int _{\R^{2d}} \widehat{{\eta}_{\frac{1}{k}}}(u,v) e^{i(uA+v B)}f \ud u \ud v 
\\ & = \frac1{(2\pi)^d}\int _{\R^{2d}} k^{2d}\widehat{\eta}(ku,kv) e^{i(uA+v B)}f \ud u \ud v
\\ & = \frac1{(2\pi)^d}\int _{\R^{2d}}
\widehat{\eta}(u,v) e^{i(\frac{u}{k}A+\frac{v}{k}B)}f \ud u\ud v 
\underset{k \to \infty}{\longrightarrow} \eta(0,0)f = f.
\end{align*}
\end{proof}

\begin{proof}[Proof of Lemma \ref{lem:McI}] 
Fix $N\in \N$, let $m= m(d,N)$ be as in Lemma \ref{lem:appmoy} (where we take $M= N$), 
and suppose 
$(A,B)$ has a bounded Weyl calculus of type $(-N,m)$. Let $(a_n)_{n\ge 1}$ be a sequence of Schwartz functions
satisfying the assumptions (i) and (ii) in the statement of the lemma.
Let $\eta \in C^{\infty}_{\rm c}(\R^{2d})$ be supported in $B(0,2)$ and identically $1$ on $B(0,1)$. Fixing $f \in X$ and $\varepsilon >0$, by Lemma \ref{lem:appid} and the uniform boundedness of the operators $a_{n}(A,B)$
we may choose a large enough integer $k$ so that 
\begin{align}\label{eq:limsup}
\underset{n \to \infty}{\lim\sup}\|a_{n}(A,B)f\| \le \underset{n \to \infty}{\lim\sup}\|a_{n}(A,B)\eta_{\frac{1}{k}}(A,B)f\| +\varepsilon.
\end{align}
Fix $n\ge 1$ for the moment. By Lemma \ref{lem:appmoy}, 
\begin{equation}\label{eq:aeta}
\begin{aligned}
\|& a_{n}(A,B) \eta_{\frac{1}{k}}(A,B)\|  
\\ & \quad \lesssim
\Big\n\sum_{|\alpha|_{\infty}\leq {N}}c_\alpha \partial^{\alpha}(a_{n}\eta_{\frac{1}{k}})(A,B) \Big\| + \|r_{a_{n},\eta_{\frac{1}{k}};N+1}(A,B)\|
\\ & \quad \lesssim  
\max_{\substack{\alpha\in \N^{2d} \\|\alpha|_{\infty}\leq {N}}}
\|\partial^{\alpha}(a_{n}\eta_{\frac{1}{k}})(A,B)\| 
+ \max_{\substack{\alpha',\beta'\in \N^{d} \\|\alpha'|,|\beta'|\leq m}} \underset{(x,\xi) \in B(0,2k)}{\sup}
\lb \xi\rb ^{|\alpha'|} |\partial^{\alpha'} _{\xi} \partial^{\beta'} _{x} a_{n}(x,\xi)|.
\end{aligned}
\end{equation}
with constants independent of $n$. 
For later reference (we don't need this here)
we observe that the constants are also uniform in $k$, as is evident from the proof of Lemma \ref{lem:appmoy}.

The first term on the right-hand side of \eqref{eq:aeta} can be estimated as follows:
\begin{align*}
\underset{|\alpha|_{\infty}\leq {N}}{\max} 
\|\partial^{\alpha}(a_{n}\eta_{\frac{1}{k}})(A,B)\| 
 & \lesssim
\underset{|\alpha|_{\infty}\leq {N}}{\max} 
\|\widehat{\partial^{\alpha}(a_{n}\eta_{\frac{1}{k}})}\|_{1}
\\ & \lesssim 
\underset{|\alpha|_{\infty}\leq {N}}{\max} 
\|(u,v) \mapsto \lb (u,v) \rb^{2d+1}\widehat{\partial^{\alpha}(a_{n}\eta_{\frac{1}{k}})}\|_{\infty}
\\ & \lesssim 
\underset{|\beta|_{\infty}\leq {N}+2d+1}{\max} 
\|\partial^{\beta}(a_{n}\eta_{\frac{1}{k}})\|_{1}
\\ & \lesssim 
\underset{|\beta|_{\infty}\leq {N}+2d+1}{\max} \|\partial^{\beta}a_{n}\|_{L^{\infty}(B(0,2k))},
\end{align*}
with constants independent of $n$.
This results in the estimate
\begin{align*}
\ & \|a_{n}(A,B) \eta_{\frac{1}{k}}(A,B)\| 
\\ & \ \lesssim
\underset{|\beta|_{\infty}\leq {N}+2d+1}{\max} \|\partial^{\beta}a_{n}\|_{L^{\infty}(B(0,2k))}
+ \underset{|\alpha'|,|\beta'|\leq m}{\max}\ \underset{(x,\xi) \in B(0,2k)}{\sup}
\lb \xi\rb ^{|\alpha'|} |\partial^{\alpha'} _{\xi} \partial^{\beta'} _{x} a_{n}(x,\xi)| 
\end{align*}
with constants independent of $n$.

Set $M:= \max(dN+d+2d^{2},2m)$; the extra factor $d$ in the first term in the maximum comes from $|\alpha|\le d|\alpha|_\infty$. If all partial derivatives up to order $M$ tend to $0$ uniformly
on $B(0,2k)$, it follows that 
$ \underset{n \to \infty}{\lim} \|a_{n}(A,B)f\| = 0$.
\end{proof}

\begin{definition}\label{def:S-N}
A function $a\in C^\infty(\R^{2d})$ is said to belong to the standard symbol class $S^{-N}$, with $N\in\Z$, if 
$$
\underset{(x,\xi) \in\R^{2d}}{\sup}\lb \xi\rb^{N+|\alpha|}\ | \partial_{\xi} ^{\alpha}\partial_{x}^{\beta} a(x,\xi)| <\infty
$$
for all multi-indices $\alpha,\beta \in \N^{d}$.
\end{definition}

The Schwartz class is included in $S^{0}$, and if $N\ge M$ then $S^{-N}\subseteq S^{-M}$.
The class $S^{-N}$ for $N>0$ plays a key role in estimating error terms that arise from the difference between the pointwise product of functions and their Moyal product. In particular, we use the fact that, for any $N>0$ and $r \in S^{-N}$, 
we may write
\begin{align}\label{eq:neg1}
T_{r}f(x) = \int _{\R^{d}} K_{r}(x,x-y)f(y)\ud y,
\end{align}
with
\begin{equation}\label{eq:neg2}
\begin{aligned}
\ & \underset{x \in \R^{d}}{\sup} \int  _{\R^{d}} |K_{r}(x,x-y)|\ud y + 
\underset{y \in \R^{d}}{\sup} \int  _{\R^{d}} |K_{r}(x,x-y)|\ud x 
\\ & \qquad\qquad\qquad  \lesssim
\underset{|\alpha|,|\beta|\leq 2d+1}{\max}\ \underset{(x,\xi) \in \R^{2d}}{\sup}
\lb \xi\rb^{|\alpha|}|\partial_{\xi} ^{\alpha} \partial_{x} ^{\beta} r(x,\xi)|.
\end{aligned}
\end{equation}
This is proven by combining \cite[Proposition 1 page 554]{Stein} and \cite[Theorem 5.12]{abels} (see also \cite[Theorem 5.15, Corollary 5.16]{abels}).

We are now ready to state and prove the main result of this section. It asserts that 
the calculus of a Weyl pair with bounded calculus of type $(-N,m)$ extends continuously to
symbols in the class $S^{-N}$:

\begin{theorem}\label{thm:Weyl-S0}
Let $N \in \N$.
If $(A,B)$ has a bounded Weyl calculus of type $(-N,m)$, where $m=m(d,N)$ is as in Lemma \ref{lem:appmoy}, then the Weyl calculus $a\mapsto a(A,B)$ extends continuously to functions $a\in S^{-N}$. More precisely, if  $a\in S^{-N}$ is given and $(a_{n})_{n \in \N}$ is sequence in $\mathcal{S}(\R^{2d})$ such that for all multi-indices $\gamma \in \N^{2d}$ we have  $$\partial^{\gamma}a_{n} \to \partial^\gamma a $$
 uniformly on compact sets as $n\to\infty$, then 
 the limit $$a(A,B):= \limn a_{n}(A,B)$$ exists in the strong operator topology of $\calL(X)$ and is independent of the approximating sequence.
Furthermore, 
for all $a \in S^{-N}$ we have 
$$ \|a(A,B)\| \lesssim
\underset{|\alpha|,|\beta|\le m+N}{\max}\ \underset{(x,\xi) \in\R^{2d}}{\sup}\lb \xi\rb^{N+|\alpha|}|\partial_{\xi} ^{\alpha}\partial_{x}^{\beta}  a(x,\xi)|.$$
\end{theorem}
\begin{proof}
The existence and uniqueness of the strong operator limits follows from what we have already proved.
As pointed out in \cite[Section 1.4, page 232]{Stein}, it is possible to approximate functions $a\in S^{-N}$ by Schwartz functions in the way stated, 
by taking $a_{n}(x,\xi) = a(x,\xi)\eta(\frac{x}{n},\frac{\xi}{n}) = a(x,\xi)\eta_{\frac1n}(x,\xi) $ for some $\eta \in C^{\infty}_{\rm c}(\R^{2d})$ such that $\eta(0,0)=1$. 

It remains to prove the bound for the norm of $a(A,B)$.
For this we return to \eqref{eq:limsup} and \eqref{eq:aeta}, both of which 
also hold if we replace $a_n$ by $a$. For a given $\eps>0$, and a large enough $k$, this gives 
\begin{align*} \|a(A,B)\| 
&  \le\| (a\eta_{\frac{1}{k}})(A,B)\| + 2\eps
\\ & \lesssim  \max_{\substack{\alpha\in \N^{2d} \\|\alpha|_{\infty}\leq {N}}}
\|\partial^{\alpha}(a\eta_{\frac{1}{k}})(A,B)\| 
\\ & 
\qquad + \max_{\substack{\alpha',\beta'\in \N^{d} \\|\alpha'|,|\beta'|\leq m}} \underset{(x,\xi) \in B(0,2k)}{\sup}
\lb \xi\rb ^{|\alpha'|} |\partial^{\alpha'} _{\xi} \partial^{\beta'} _{x} a(x,\xi)|
 +2\eps
\end{align*}
with estimates uniform in $\eps>0$ and $k\ge 1$ (note that the sup norms of the derivatives of $\eta_k$ are uniform in $k\ge 1$).

Each expression in the first term on the right-hand side can be estimated using the type $(-N,m)$ of the Weyl calculus of $(A,B)$: 
\begin{align*}
\|\partial^{\alpha}(a\eta_{\frac{1}{k}})(A,B)\| 
& \lesssim 
\underset{\substack{\gamma,\delta\in \N^{d} \\|\gamma|,|\delta|\le m}}{\max}\ \underset{(x,\xi) \in\R^{2d}}{\sup}\ \lb \xi\rb^{N+|\gamma|}| \partial_{\xi} ^{\gamma}\partial_{x}^{\delta}\partial^\alpha (a\eta_{\frac{1}{k}})(x,\xi)|
\\ & \lesssim
\underset{\substack{\alpha',\beta'\in \N^{d} \\ |\alpha'|,|\beta'|\le m+Nd}}{\max}\ \underset{(x,\xi) \in\R^{2d}}{\sup}\ \lb \xi\rb^{N+|\alpha'|}| \partial_{\xi} ^{\alpha'}\partial_{x}^{\beta'} 
a(x,\xi)|,
\end{align*}
again with estimates uniform in $\eps>0$ and $k\ge 1$.
Since $\eps>0$ was arbitrary, this results in the desired estimate. 
\end{proof}

\subsection{Bounded Weyl calculus of type $0$ for Banach space-valued standard pairs}\label{sec:stand}

Let $X$ be a UMD space. On $L^p(\R^d;X)$, $1<p<\infty$, we consider the vector-valued standard pair $(Q \otimes I_{X},P \otimes I_{X})$ defined by
$Q \otimes I_{X} = (Q_j\otimes I_{X})_{j=1} ^{d}$ and
$P \otimes I_{X} = (P_j \otimes I_{X})_{j=1} ^{d}$,
where $Q_j$ and $P_j$ are the position and momentum operators as in Example \ref{ex:HO}.
Note that $(Q \otimes I_{X},P \otimes I_{X})$ is a Weyl pair: as in the scalar case,
$iQ_j \otimes I_{X}$ and $iP_j \otimes I_{X}$ generate multiplication and translation groups on $L^p(\R^d;X)$ given by the same formulas as in the scalar-valued case (Example \ref{ex:HO}).
The commutation relations for the vector-valued extensions also follow from their scalar-valued counterparts. 

\medskip\noindent
{\em Notation}. \ In order to simplify notation we will suppress the tensors with $I_X$ when no confusion is likely to arise.

\medskip
As an illustration of Definition \ref{def:typeAB} we now prove:

\begin{theorem}\label{thm:type0} 
If $X$ is a UMD Banach space, the standard pair $(Q,P)$ has a bounded Weyl calculus of type $0$ on $L^p(\R^d;X)$ for all  $1<p<\infty$.
\end{theorem}

To prove this theorem we will use \cite[Theorem 6]{pz}. To do so, we need to view $a(Q,P)$ as a pseudo-differential operator acting on $L^{2}(\R^{d};X)$. This is possible thanks to the following lemma.

\begin{lemma}
\label{lem:weyltopseudo}
For every $a \in \mathcal{S}(\R^{2d})$ there exists a unique $b \in \mathcal{S}(\R^{2d})$ such that $a(Q,P) = T_{b}$, where $T_{b}$ is the pseudo-differential operator on $L^2(\R^d)$ defined by
$$
T_{b}f(x) =\frac1{(2\pi)^{d/2}} \int  _{\R^{d}} b(x,\xi)\widehat{f}(\xi)e^{i\xi x}\ud\xi.
$$
This function is given by 
\begin{align}\label{eq:b-asympt}
b(x,\xi) = \sum _{|\alpha|=1} \frac1{\alpha!}\frac1{i^{|\alpha|}}\partial_\xi^{\alpha}\partial_y^{\alpha}p_a(x,\xi,y,\xi')\Big|_{y=x, \xi'=\xi}
+r_{a}(x,\xi),
\end{align}
where $r_{a}\in \mathscr{S}(\R^{2d})$ and $p_a(x,\xi,y,\xi') = a(\frac{x+y}{2},\xi)$.
Moreover, for all $m \in \N$, there exists $\tilde{m} \geq m$, depending only on $m$ and $d$, such that 
$$\underset{|\alpha|,|\beta|\leq m}{\max}\ \underset{(x,\xi) \in \R^{2d}}{\sup}\
\lb \xi\rb^{|\alpha|}|\partial_{\xi} ^{\alpha} \partial_{x} ^{\beta} r_a(x,\xi)|\lesssim \underset{|\alpha|,|\beta|\leq \tilde{m}}{\max}\ \underset{(x,\xi) \in \R^{2d}}{\sup}\
\lb \xi\rb^{|\alpha|}|\partial_{\xi} ^{\alpha} \partial_{x} ^{\beta} a(x,\xi)|.$$
\end{lemma}
\begin{proof}
The first assertion follows from \cite[Proposition 1, page 554]{Stein} (see also \cite[formula (58), page 258]{Stein}).
As in the proof of Lemma \ref{lem:appmoy}, the estimate follows from \cite[Theorem 3.15]{abels}. 
\end{proof}

\begin{proof}[Proof of Theorem \ref{thm:type0}]

We must show that there exists an integer $m\in \N$ such that
for all $a \in \mathcal{S}(\R^{2d})$ we have
$$
\|a(Q,P)\|_{\calL(L^p(\R^d;X))}\lesssim \underset{|\alpha|,|\beta|\le m}{\max}\ \underset{(x,\xi) \in\R^{2d}}{\sup}\lb \xi\rb^{|\alpha|}| \partial_{\xi} ^{\alpha}\partial_{x}^{\beta} a(x,\xi)|.
$$

Let $a \in \mathscr{S}(\R^{2d})$. We first apply Lemma \ref{lem:weyltopseudo} to write
\begin{equation}\label{eq:aPQX}
\begin{aligned}
a(P,Q) 
  = T_{b} 
  = \sum  _{|\alpha|=1} \frac1{\alpha!}\frac1{i^{|\alpha|}}T_{\partial_{\xi} ^{\alpha} \partial_{x} ^{\alpha} p_a}+ T_{r_{a}},
\end{aligned}
\end{equation}
where $\partial_{\xi} ^{\alpha} \partial_{x} ^{\beta} p_a(x,\xi)$
is short-hand for the expression 
$\partial_\xi^{\alpha}\partial_y^{\alpha}p_a(x,\xi,y,\xi')|_{y=x, \xi'=\xi}$
occurring in \eqref{eq:b-asympt}.
We now estimate the $L^p(\R^d;X)$-norms of the terms on the right-hand side of \eqref{eq:aPQX} separately, starting with $T_{r_{a}}$. As pointed out in \eqref{eq:neg1} and \eqref{eq:neg2} we have
\begin{align*}
T_{r_{a}}f(x) & =\int _{\R^{d}} K_{r_a}(x,y)f(y)\ud y
\end{align*}
with 
\begin{align*}
\ & \underset{x \in \R^{d}}{\sup} \int  _{\R^{d}} |K_{r_a}(x,y)|\ud y + 
\underset{y \in \R^{d}}{\sup} \int  _{\R^{d}} |K_{r_a}(x,y)|\ud x 
\\ & \qquad\qquad \lesssim
\underset{|\alpha|,|\beta|\leq 2d+1}{\max}\ \underset{(x,\xi) \in \R^{2d}}{\sup}
\lb \xi\rb^{|\alpha|}|\partial_{\xi} ^{\alpha} \partial_{x} ^{\beta} r_{a}(x,\xi)|.
\end{align*}
Therefore, by Schur's lemma (in the formulation of \cite[Lemma 4.1 with $p=q$, $r=1$, $\phi=\psi\equiv 1$]{NP}, noting that the proof extends without change to the vector-valued case), 
$T_{r_{a}}$ extends to a bounded operator on $L^p(\R^d;X)$ of norm at most
\begin{equation}\label{eq:est-Tr}
\begin{aligned} \n T_r\n_{\calL(L^p(\R^d;X))}
&\lesssim \underset{|\alpha|,|\beta| \leq 2d+1}{\max}\ \underset{(x,\xi) \in \R^{2d}}{\sup}
\lb \xi\rb^{|\alpha|}|\partial_{\xi} ^{\alpha} \partial_{x} ^{\beta} r_{a}(x,\xi)|
\\ & \lesssim 
\underset{|\alpha|,|\beta| \leq \tilde{m}}{\max}\ \underset{(x,\xi) \in \R^{2d}}{\sup}
\lb \xi\rb^{|\alpha|}|\partial_{\xi} ^{\alpha} \partial_{x} ^{\beta} {a}(x,\xi)|,
\end{aligned}
\end{equation}
for some $\tilde{m} \geq 2d+1$, the second inequality being a consequence of  Lemma \ref{lem:weyltopseudo}. 

Next we estimate the $L^p(\R^d;X)$-norms of the operators $T_{\partial_{\xi} ^{\alpha} \partial_{x} ^{\beta} p_a}$.
Let $\alpha,\beta \in \N^{d}$ be such that 
$|\alpha|,|\beta| \leq 1$.
In order to apply \cite[Theorem 6]{pz},
 we remark that $p_{a,\alpha,\beta}:=\partial^{\alpha} _{\xi} \partial ^{\beta} _{x}p_a$
has the following (trivial) properties: 
\begin{enumerate}
 \item[\rm(a)] for all $|\gamma|\le 2d+5$ 
 and $x\in \R^d$ we have 
 \begin{align*}
 \lb \xi\rb^{|\gamma|} |\partial_{\xi} ^{\gamma} p_{a,\alpha,\beta}(x,\xi)|
 & = \lb \xi\rb^{|\gamma|} |\partial_{\xi} ^{\alpha+\gamma} p_a(x,\xi)|
 \\ & \le 
 \underset{|\alpha'|,|\beta'| \leq 2d+6}{\max}\ \underset{(x,\xi) \in \R^{2d}}{\sup}
\lb \xi\rb^{|\alpha'|}|\partial_{\xi} ^{\alpha'} \partial_{x} ^{\beta'} a(x,\xi)|
\end{align*}
\item[\rm(b)] for all $|\gamma|,|\delta|\le 2d+5$ we have 
\begin{align*} 
|\partial_{\xi} ^{\gamma} \partial_{x} ^{\delta} p_{a,\alpha,\beta}(x,\xi)|
 & = |\partial_{\xi} ^{\alpha+\gamma} \partial_{x} ^{\beta+\delta} a(x,\xi)|
 \\ & \le 
 \underset{|\alpha'|,|\beta'| \leq 2d+6}{\max}\ \underset{(x,\xi) \in \R^{2d}}{\sup}
\lb \xi\rb^{|\alpha'|}|\partial_{\xi} ^{\alpha'} \partial_{x} ^{\beta'} a(x,\xi)|.
\end{align*}
\end{enumerate}
This means that each $b_{\alpha,\beta}$
 belongs to the class $S^{0}_{1,0}(2d+5,X)$ as defined in \cite[Definition 3]{pz} 
(note that the $R$-boundedness condition in this definition reduces to a uniform boundedness
condition in view of the fact that we are considering scalar-valued symbols).
Therefore, by \cite[Theorem 6]{pz}
(and its proof, which shows that the estimates depend linearly on the expressions on the right-hand sides in (a) and (b)),
the operators $T_{\partial_{\xi} ^{\alpha} \partial_{x} ^{\beta} a}$ are bounded on $L^{p}(\R^d;X)$, and \begin{equation}\label{eq:est-Tpartial}
\begin{aligned}
\|T_{\partial_{\xi} ^{\alpha} \partial_{x} ^{\beta} a}\|_{\calL(L^{p}(\R^d;X))} &
\lesssim \underset{|\alpha|,|\beta| \leq 2d+6}{\max}\ \underset{(x,\xi) \in \R^{2d}}{\sup}
\lb \xi\rb^{|\alpha|}|\partial_{\xi} ^{\alpha} \partial_{x} ^{\beta} a(x,\xi)|.
\end{aligned}
\end{equation}
Putting together the estimates \eqref{eq:est-Tr} and \eqref{eq:est-Tpartial} we obtain
$$
\|a(P,Q)\|_{\calL(L^{p}(\R^d;X))} \lesssim
 \underset{|\alpha|,|\beta| \leq \max(\tilde{m},2d+6)}{\max}\ \underset{(x,\xi) \in \R^{2d}}{\sup}
\lb \xi\rb^{|\alpha|}|\partial_{\xi} ^{\alpha} \partial_{x} ^{\beta} a(x,\xi)|,
$$
which concludes the proof.
\end{proof}

\section{The operator $A^2+B^2$}\label{sec:A2B2}

In this section we show how the Weyl calculus of the pair $(A,B)$ relates to the functional calculus of the operator $A^2+B^2$.

For Weyl pairs $(A,B)$ of dimension $d$ we define $$A^2 := \sum_{j=1}^d A_j^2, \quad B^2 := \sum_{j=1}^d B_j^2$$
with domains $\D(A^2) := \bigcap_{j=1}^d \Dom(A_j^2)$ and $\D(B^2) := \bigcap_{j=1}^d \Dom(B_j^2)$. The operator $A^2+B^2$ is understood as being defined on $\D(A^2)\cap \D(B^2)$.
Earlier we have already defined $\D(A) := \bigcap_{j=1}^d \Dom(A_j)$ and $\D(B) := \bigcap_{j=1}^d \Dom(B_j)$. 

The following proposition is an immediate consequence of Lemmas \ref{lem:Kato} and \ref{lem:D}.

\begin{proposition}\label{prop:dd}
If $(A,B)$ is a Weyl pair of dimension $d$ on $X$, then $\D(A^2)\cap \D(B^2)$ is dense in $X$ and invariant under the groups $(e^{itA_j})_{t\in\R}$ and $(e^{itB_j})_{t\in\R}$, $1\le j\le d$.
\end{proposition}

The next theorem shows, among other things, that for any Weyl pair $(A,B)$ the operator $-(A^2+B^2)$ is closable and 
its closure generates an analytic $C_0$-semigroup of angle $\frac12\pi$. Up to a scaling, this semigroup can be thought of as an abstract
version of the Ornstein--Uhlenbeck semigroup. 
For the standard pair, such a theorem is well-known to mathematical physicists, going back at least, to \cite{unter}. It was rediscovered for the Ornstein-Uhlenbeck semigroup in \cite[Theorem 3.1]{NP}. Here we prove that it holds for all Weyl pairs.

\begin{theorem}\label{thm:A2B2} Let $(A,B)$ be a Weyl pair. The operators 
 $$ P(t) := \Big(1+\frac{1-e^{-t}}{1+e^{-t}}\Bigr)^d\exp\Bigl(-\frac{1-e^{-t}}{1+e^{-t}}(A^2+B^2)\Bigr) \quad (t \ge 0)
$$
define a uniformly bounded $C_0$-semigroup on $X$.
The dense set $\D(A^2)\cap \D(B^2)$ is a core for its generator $-L$, and, on this core, we have the identity
$$L =\frac12(A^2+B^2)- \frac12d.$$
The semigroup $(P(t))_{t\ge 0}$ extends to an analytic semigroup of angle $\frac12\pi$
that is uniformly bounded and strongly continuous on every subsector of smaller angle.
\end{theorem}
In the above formula for $P(t)$, for $t>0$ the right-hand side is interpreted in terms of the Weyl calculus for the pair $(A,B)$,
i.e., $P(t) = a_{t}(A,B)$, where 
\begin{align}\label{def:a}a_{t}(x,\xi):=   (1+\lambda_t)^d e^{-\lambda_t(|x|^2+|\xi|^2)}
\end{align}
with $\lambda_t =  \frac{1-e^{-t}}{1+e^{-t}}$. For $t=0$ we interpret the formula as stating that $P(0)=I$.
\begin{proof} 
The semigroup property $P(t_1)P(t_2) = P(t_1+t_2)$ 
follows from the identity 
\begin{equation}\label{eq:asas} 
\begin{aligned} 
a_{{t_1}} \#\, a_{{t_2}} (x, \xi) 
 & = \frac{1}{\pi^{2d}} (1+\lambda_{t_1})^d(1+\lambda_{t_2})^d
\\ & \qquad \times \int_{\R^{2d}}\int_{\R^{2d}} e^{-\lambda_{t_1}((x+u)^2+( \xi+u')^2)} e^{-\lambda_{t_2}((x+v)^2+ ( \xi+v')^2)} 
\\ & = (1+\lambda_{t_1+t_2})^d e^{-\lambda_{t_1+t_2}(|x|^2+ |\xi|^2)}
\\ & = a_{{t_1+t_2}}(x, \xi)
\end{aligned}
\end{equation}
which is obtained by elementary computation.

Next we prove the strong continuity $\lim_{t\downarrow 0}P(t)f=f$ for all $f\in X$. 
Fix $t>0$ for the moment. We have
\begin{equation}\label{eq:Pt1}
\begin{aligned}
 P(t)f & = a_{t}(A,B)f 
 \\ & =  \frac1{(2\pi)^{d}}\int_{\R^{2d}} \wh a_t(u,v) e^{i(uA+v B)}f\ud u\ud v
 \\ & =  \frac1{(2\pi)^d}(1+\lambda_t)^d \frac1{(2\lambda_t)^d}\int_{\R^{2d}} \exp\Bigl( 
-\frac{1}{4\lambda_t}(|u|^2+|v|^2) \Bigr)e^{i(uA+v B)}f\ud u\ud v
\\ & =  \frac1{(2\pi)^d}\frac{1}{(1-e^{-t})^d} \int_{\R^{2d}} \exp\Bigl( 
-\frac{1+e^{-t}}{4(1-e^{-t})}(|u|^2+|v|^2) \Bigr)e^{i(uA+v B)}f\ud u\ud v\end{aligned}
\end{equation}
so that 
\begin{align}\label{eq:Pt2} \n P(t)\n \le \frac{M_AM_B}{(2\pi)^{d}} \frac{1}{(1-e^{-t})^d}\int_{\R^{2d}}e^{ 
-\frac{1}{4(1-e^{-t})}(|u|^2+|v|^2)}\ud u\ud v \lesssim M_AM_B. \end{align}
This proves the uniform boundedness of $P(t)$ for $t>0$. Strong continuity follows from the fact that 
$\wh a_t \to \delta_0$ weakly (in the sense that we have strong convergence against every $f\in C_{\rm b}(\R;X)$).

Let us denote the generator of the $C_0$-semigroup $(P(t))_{t\ge 0}$ by $-L$. We claim that $Lf = \frac12df-\frac12(A^2 f+B^2f)$
for all $f\in \D(A^2)\cap \D(B^2)$. 
Our argument will be somewhat formal. The reader will have no difficulty in making it rigorous by proceeding as follows:
write 
$$ e^{i(uA+vB)}f = \psi(u,v) e^{i(uA+vB)}f + (1-\psi(u,v)) e^{i(uA+vB)}f$$
for some compactly supported smooth function $a\psi$ which equals $1$ in a neighbourhood of $(0,0)$. Treating the resulting integrals separately,
the ones involving $\psi$ will give the desired convergence while the ones involving $1-\psi$ will vanish as we pass to the limit.

Proceeding to the details, we write 
$P(t) = (1+\lambda)^d R(\lambda)$, where $\lambda = \lambda_t = \frac{1-e^{-t}}{1+e^{-t}}.$
Then, 
\begin{align*} 
\frac{\dd }{\dd t}P(t)f  = \frac{\dd }{\dd \lambda}[(1+\lambda)]^d R(\lambda)f] \frac{\dd \lambda}{\dd t}  = \frac12(1-\lambda^2)\frac{\dd }{\dd \lambda}[(1+\lambda)]^d R(\lambda)f] .
\end{align*}
In the limit $t\downarrow 0$ we also have $\lambda\downarrow 0$ and  $\frac12(1-\lambda^2) \to \frac12$.
Hence the claim will be proved if we show that $\frac{\dd }{\dd \lambda}R(\lambda)f\to df- (A^2+B^2) f$
for $f\in \D(A^2)\cap \D(B^2)$.
We have
\begin{align*} 
\ & \lim_{\lambda\downarrow 0} \frac{\dd }{\dd \lambda}[(1+\lambda)]^d R(\lambda)f]  \\
&  \ \ = \lim_{\lambda\downarrow 0}\frac{\dd }{\dd \lambda}\Bigl[(1+\lambda)^d  \frac1{(2\pi)^{d}}\int_{\R^{2d}}\int_{\R^{2d}} \widehat{e^{-\lambda(|u|^2+|v|^2)}} e^{i(uA+v B)}f\ud u\ud v \Bigr] \\
&  \ \ = \lim_{\lambda\downarrow 0} d(1+\lambda)^{d-1} \frac1{(2\pi)^{d}}\int_{\R^{2d}}\int_{\R^{2d}} \widehat{ e^{-\lambda(|u|^2+|v|^2)}} e^{i(uA+v B)}f\ud u\ud v 
\\  & \ \ \quad + \lim_{\lambda\downarrow 0} \Bigl[(1+\lambda)^d \frac1{(2\pi)^{d}}\int_{\R^{2d}}\int_{\R^{2d}} \frac{\dd }{\dd \lambda}\widehat{ e^{-\lambda(|u|^2+|v|^2)}} e^{i(uA+vB)}f\ud u\ud v \Bigr]
\\ & \ \ = (I)+(II).
\end{align*}
Now, for any $f\in X$, 
\begin{align*} 
(I) &  =  \lim_{\lambda\downarrow 0} \frac{d}{(2\pi)^{d}}\int_{\R^{2d}}\int_{\R^{2d}} \widehat{ e^{-\lambda(|u|^2+|v|^2)}} e^{i(uA+vB)}f\ud u\ud v \\ & = d \int_{\R^{2d}}\int_{\R^{2d}} \delta_{(0,0)} e^{i(uA+vB)}f\ud u\ud v = df.
\intertext{
Similarly, for $f\in \D(A^2)\cap \D(B^2)$,
          }
(II)   & = \lim_{\lambda\downarrow 0} \frac1{(2\pi)^{d}}\int_{\R^{2d}}\int_{\R^{2d}} -\widehat{(|u|^2+|v|^2) e^{-\lambda(|u|^2+|v|^2)}} e^{i(uA+vB)}f\ud u\ud v 
\\ & = \int_{\R^{2d}}\int_{\R^{2d}} \Delta \delta_{(0,0)} e^{i(uA+vB)}f\ud u\ud v = -(A^2  + B^2) f.
\end{align*} 
Here, $\Delta \delta_{(0,0)}$ denotes the Laplacian of the Dirac delta function in the sense of distributions.

We will prove next that $\D(A^2)\cap \D(B^2)$ is a core for $L$.
We have already seen that $\D(A^2)\cap \D(B^2)$ is contained in $\Dom(L)$. The definition of the operators $P(t)$ together with 
the commutation relation defining Weyl pairs implies that $\D(A^2)\cap \D(B^2)$ is invariant under $P(t)$.
Since $\D(A^2)\cap \D(B^2)$ is also dense in $X$, a standard result in semigroup theory 
implies that 
$\D(A^2)\cap \D(B^2)$ is a core for $L$.

To complete the proof it remains to show the final assertion. By a standard analytic extension argument, the right-hand side of \eqref{eq:Pt1} defines an
analytic extension of $P(t)$ to the open right half-plane which again satisfies the semigroup property.
Estimating as in \eqref{eq:Pt2} we see that this extension is uniformly bounded on every sector of
angle strictly less than $\frac12\pi$. A standard semigroup argument (see, e.g., \cite[Exercise 9.8]{Haase-ISEM})
gives the strong continuity of the extension on each of these sectors.
\end{proof}

\begin{example}\label{ex:HA}
For the standard pair of momentum and position we recover the standard fact that the harmonic oscillator defined by $-Lf(u) = \frac12\Delta f(u) - \frac12|u|^2f(u)$ generates a holomorphic semigroup of angle $\frac12\pi$, strongly continuous on each smaller sector, on each of the spaces on $L^p(\R^d)$ with $1\le p<\infty$.
\end{example}

For later use we make the following simple observation.

\begin{corollary}\label{cor:spectral}
 For all $t>0$ we have $$ \n t L P(t)\n \le 2^{d+2}d M_A M_B(1+t) e^{-t} .$$
\end{corollary}
\begin{proof}
Using the same notation as before, write $\lambda_t' = \frac{2e^{-t}}{(1+e^{-t})^2}$ for the derivative of $t\mapsto \lambda_t = \frac{1-e^{-t}}{1+e^{-t}}.$ In view of 
$LP(t)f =  -\frac{\rm d}{{\rm d}t}P(t)f$, differentiation of the right-hand side of \eqref{eq:Pt1} (and noting that $\frac{1}{1-e^{-t}} = \frac12( 1+\lambda_t^{-1}) $)
gives 
\begin{equation*}
\begin{aligned}
 & (4\pi)^d LP(t)f 
 \\ & \quad =  -\frac{\rm d}{{\rm d}t}  \Bigl((1+\lambda_t^{-1})^d\int_{\R^{2d}} \exp(-(|u|^2+|v|^2)/4\lambda_t) e^{i(uA+vB)}f\ud u\ud v\Bigr) 
\\ & \quad = d(1+\lambda_t^{-1})^{d-1}\frac{\lambda_t'}{\lambda_t^2} \int_{\R^{2d}} \exp(-(|u|^2+|v|^2)/4\lambda_t) e^{i(uA+vB)}f\ud u\ud v
\\ & \quad\quad   -(1+\lambda_t^{-1})^d \frac{\lambda_t'}{4\lambda_t^{2}}
\int_{\R^{2d}} (|u|^2+|v|^2)\exp(-(|u|^2+|v|^2)/4\lambda_t) e^{i(uA+vB)}f\ud u\ud v
\end{aligned}
\end{equation*}
and therefore 
\begin{equation*}
\begin{aligned}
 &\n LP(t)f\n  
\\ & \quad \le d (1+\lambda_t)^{d-1}\frac{\lambda_t'}{\lambda_t} \frac{M_A M_B \n f\n}{(4\pi \lambda_t)^d}\int_{\R^{2d}} \exp(-(|u|^2+|v|^2)/4\lambda_t) \ud u\ud v
\\ & \quad \quad + (1+\lambda_t)^d
\frac{\lambda_t'}{4\lambda_t^{2}}
\frac{M_A M_B \n f\n}{(4\pi \lambda_t)^d} \int_{\R^{2d}} (|u|^2+|v|^2)\exp(-(|u|^2+|v|^2)/4\lambda_t)\ud u\ud v.
\end{aligned}
\end{equation*}
In view of the identities
$$  \frac1{(4\pi \lambda_t)^d}\int_{\R^{2d}} \exp(-(|u|^2+|v|^2)/4\lambda_t) \ud u\ud v = 1$$
and 
\begin{align*}
& \frac1{(4\pi \lambda_t)^d}\int_{\R^{2d}}  (|u|^2+|v|^2)\exp(-(|u|^2+|v|^2)/4\lambda_t) \ud u\ud v
\\ & \quad = \frac1{(4\pi \lambda_t)^d}\sum_{j=1}^{2d}\int_{\R^{2d}}  w_j^2\exp(-|w|^2/4\lambda_t) \ud w
\\ & \quad = \frac1{(4\pi \lambda_t)^d}\sum_{j=1}^{2d}\Bigl(\int_{\R}  w_j^2\exp(-w_j^2/4\lambda_t) \ud w_{j}\Bigr)\prod_{\substack{1\le k\le 2d \\ k\not=j}}\int_{\R}\exp(-w_k^2/4\lambda_t) \ud w_k
\\ & \quad = \frac1{(4\pi \lambda_t)^{1/2}}\sum_{j=1}^{2d}\int_{\R}  w_j^2\exp(-w_j^2/4\lambda_t) \ud w_{j} \\ & \quad = 4d\lambda_t
\end{align*}
we obtain
\begin{align*}
 \n tLP(t)f\n & \le t\Bigl(d 2^{d-1}\frac{\lambda_t'}{\lambda_t }+  2^d\frac{\lambda_t'}{4\lambda_t^{2}}\cdot 4d\lambda_t\Bigr)M_A M_B \n f \n
  \\ & = 2^{ d+1}dt\frac{\lambda_t'}{\lambda_t }M_A M_B \n f \n
  \\ & = 2^{ d+1}dt \cdot  \frac{2e^{-t}}{(1+e^{-t})^2} \cdot \frac{1+e^{-t}}{1-e^{-t}}M_A M_B \n f\n
  \\ & =  2^{ d+2}d\frac{t}{1-e^{-2t}}e^{-t} M_A M_B \n f \n
  \\ & \le 2^{ d+2}d(1+t) e^{-t} M_A M_B \n f\n.
\end{align*}
\end{proof}

\subsection{Ground states}

Let $(A,B)$ be a Weyl pair on the Banach space $X$.
Upon passing to the limit $t_1,t_2\to\infty$ in \eqref{eq:asas} one sees that if
the limit $$a_\infty(A,B) := \lim_{t\to\infty} a_t(A,B)$$ exists
in the weak operator topology of $\calL(X)$, then it is a projection. 
That this limit indeed exists under the assumption that $X$ be reflexive 
is a consequence of the following lemma.

\begin{lemma}
  Let $(S(t))_{t\ge 0}$ be a $C_0$-semigroup on a reflexive Banach space $X$ and let 
  $(T_t)_{t\ge 0}$ be a uniformly bounded family of operators on $X$ such that 
  $ S(s)\circ T_t = T_t \circ S(s) = T_{t+s}$ for all $s,t\ge 0$. 
  Then there exists a bounded operator $\pi$ on $X$ such that $\lim_{t\to\infty} T_t x =  \pi (x)$ weakly for all $x\in X$.
 \end{lemma}
 \begin{proof}
 Fix $x\in X$. Since $X$ is reflexive, any sequence $t_n\to\infty$ has a subsequence $t_{n_k}\to\infty$ such that $\limk T_{n_k}x $ exists weakly. Let $\pi(x)$ be this weak limit. We will show that $\pi(x)$ does not depend on the choice of 
 the sequence $t_n\to\infty$, nor on the choice of the weakly convergent subsequence $t_{n_k}\downarrow 0$.
 To this end it suffices to show that if both $r_k\to\infty$ and $s_k\to\infty$ are such that the weak limits
 $y:= \limk T_{r_k}x$ and $y':= \limk T_{s_k}x$ exist, then $y = y'$.
 By passing to a further subsequence we may assume that $r_k \le s_k$ for all $k$. Then
 $T_{s_k}x = S(s_k-r_k)T_{r_k}x$ and therefore, for all $x\s\in X\s$,
 \begin{align*}
  |\lb T_{s_k}x-T_{r_k}x,x\s\rb| & =
  |\lb S(s_k-r_k)T_{r_k}x-T_{r_k}x,x\s\rb|
  \\ & \le \n T_{r_k}x\n  \n S\s(s_k-r_k)x\s-x\s\n_{X^{*}}
  \\ & \le M\|x\|\n S\s(s_k-r_k)x\s-x\s\n_{X^{*}},
 \end{align*}
 where $M = \sup_{t\ge 0} \n T_t\n$.
 Since $X$ is reflexive, the adjoint semigroup $(S^*(t))_{t\ge 0}$ is strongly continuous (see \cite[Proposition I.5.14]{EngNag}),
 it follows that 
 $$ |\lb y-y',x\s\rb| = \limk |\lb T_{s_k}x-T_{r_k}x,x\s\rb| \le
 M \|x\|\limk \n S\s(s_k-r_k)x\s-x\s\n_{X^{*}} = 0.$$
 This being true for all $x\s\in X\s$, we conclude that $y=y'$.
 
 The operator $\pi$ thus defined is linear and bounded, with norm $\n \pi\n \le M$. That $\pi (x)= \lim_{t\to\infty}T_t x$ weakly
 now follows from a standard subsequence argument.
 \end{proof}

 \begin{proposition} 
 \label{prop:nullrange}
 If $(A,B)$ is a Weyl pair on a reflexive Banach space $X$, then the weak operator limit  $\pi  := \lim_{t\to\infty} a_t(A,B)$ exists in $\calL(X)$.
 Furthermore, $\Ker(L)=\Ran(\pi) $ and $\overline{\Ran(L)}= \Ker(\pi).$
 \end{proposition}
 \begin{proof}
For all $t_1,t_2\ge 0$, \eqref{eq:asas}
 implies $$e^{-t_1L}\circ a_{{t_2}}(A,B) =  a_{{t_2}}(A,B)\circ e^{-t_1L} = a_{t_1+t_2} (A,B).$$
Hence the first assertion follows from the lemma. 

The second assertion is proved by a routine semigroup argument. If $f\in \Ker(L)$, then $a_t(A,B)f = e^{-tL}f =f$ implies $\pi(f) = f$, and 
conversely if $\pi(f) = f$, then for all $t\ge 0$ and $\phi\in X^*$ we have 
\begin{align*}\lb e^{-tL}f,\phi\rb 
& = \lb f, e^{-tL^*}\phi\rb= \lim_{s\to\infty}\lb a_s(A,B)f, e^{-tL^*}\phi\rb 
\\ & = \lim_{s\to\infty} \lb a_{s+t}(A,B)f,\phi\rb = \lb \pi(f),\phi\rb = \lb f,\phi\rb
\end{align*}
 and therefore 
$e^{-tL}f=f$. This implies $f\in \Dom(L)$ and $Lf = 0$. This proves $\Ker(L)=\Ran(\pi)$. The proof that $\overline{\Ran(L)}= \Ker(\pi)$ is equally simple.
 \end{proof}

In particular we see that $X$ admits the direct sum decomposition $\Ker(L)\oplus \overline{\Ran(L)}$. Such a decomposition holds
for every sectorial operator on a reflexive Banach space (see \cite[Proposition 10.1.9]{HNVW2}); the point of the proposition is to identify the 
associated projection as being given by $\pi$.

\section{Transference}\label{sec:transf}

Let $X$ be a Banach space. For functions $a\in \mathscr{S}(\R^{2d})$, 
the {\em twisted convolution} with a function $g\in C_{\rm c}(\R^{2d};X)$ is defined by
\begin{align}\label{eq:twisted}
C_{a}g(x,\xi) :=&\,\frac1{(2\pi)^d}\int_{\R^{2d}} e^{\frac12i(x\eta-y\xi)} a(y,\eta )g(x-y,\xi-\eta )\ud y\ud \eta .
\end{align}
By the pointwise inequality $|C_a g| \le|a|*|g|$ and Young's inequality, $C_{a}$ extends to a bounded operator on $L^p(\R^{2d};X)$ for all $1\le p\le \infty$.

We begin with a Coifman--Weiss type transference result.

\begin{proposition}[Transference]\label{prop:transf} 
Let $(A,B)$ be a Weyl pair of dimension $d$ on a Banach space $X$ and set $M_A := \sup_{t\in \R} \n e^{itA}\n$ and $M_B := \sup_{t\in \R} \n e^{itB}\n$. Let $1\le p<\infty$.
\begin{enumerate}
\item For all $a\in \mathscr{S}(\R^{2d})$ we have 
\begin{align*}
\n a(A,B)\n \le M_A^2 M_B^2 \n C_{\wh a} \n_{\calL(L^p(\R^{2d};X))}.
\end{align*}
\item Let $\{a_j: \,j\in J\}$ be a family of functions in $ \mathscr{S}(\R^{2d})$. If the family of twisted convolutions $\{C_{\wh a_j},\, j\in J\}$ is $R$-bounded in $\calL(L^p(\R^{2d};X))$, then
$\{a_j(A,B):\, j\in J\}$ is $R$-bounded in $\calL(X)$, and in that case
\begin{align*}
\mathscr{R}_p(a_j(A,B):\,j\in J) \le M_A^2 M_B^2 \mathscr{R}_p(C_{\wh a_j}:\, j\in J).
\end{align*}
\item
Let $\{a_j: \,j\in J\}$ be a family of functions in $ \mathscr{S}(\R^{2d})$. If the family of twisted convolutions $\{C_{\wh a_j},\, j\in J\}$ satisfies
$$
\E \Big\|\sum_{j \in J} \eps_{j} C_{\wh a_j}g\Big\| \lesssim \|g\| \quad \forall g\in L^{p}(\R^{2d};X),
$$
then
$$\E \Big\|\sum_{j \in J} \eps_{j} a_{j}(A,B)f\Big\| \lesssim \|f\| \quad \forall f \in X.
$$
\end{enumerate}

\end{proposition}

\begin{proof}
For $r>0$ we will use the short-hand notation 
$[-r,r]^{2} = \{(x,\xi)\in \R^{2d}:\, |x|\le r, \, |\xi|\le r\}$.
The elementary estimate $  \|a(A,B)\|\leq M_AM_B \|\wh a\|_{1} $ shows that, for any given $\eps>0$,
we may choose
$N>0$ so large that the operator
\[
  a(A,B)_{(N)} f: = \frac1{(2\pi)^d}\int_{\complement [-N,N]^2} \wh a(u,v)e^{i(uA+v B)}
 f\ud u\ud v
\]
satisfies
\[
  \|a(A,B)_{(N)}\|\leq \frac{M_AM_B}{(2\pi)^d} \int_{\complement [-N,N]^2} |\wh a(u,v)| \ud u\ud v
  < \varepsilon .
\]
We will therefore concentrate on estimating the norm of the operator
\[
  a(A,B)^{(N)} f: = \frac1{(2\pi)^d}\int_{[-N,N]^2} \wh a(u,v)e^{i(uA+v B)}
 f\ud u\ud v.
\]
Accordingly set $\wh a^{(N)} := 1_{[-N,N]^{2}}\wh a$.
Choose $M$ so large that
$\frac{M+N}{M}\leq1+\varepsilon$. Let us write $U(u,v) = e^{i(uA+v B)}$ for brevity.  
By \eqref{eq:sigma} we have 
$ U(u,v)\circ U(-u,-v) = I$
and therefore
\begin{equation}
  \|f\| \leq M_AM_B\|U(-u,-v)f\|, \quad f\in X.
\end{equation}
Averaging over $[-M,M]^2$, for all $1\le p<\infty$ and $f\in X$ we obtain 
\begin{align*}
\! &  \| a(A,B)^{(N)} f\|^p\\ &\leq \frac{M_A^pM_B^p}{(2M)^{2d}}\int_{[-M,M]^2} \|U(-u,-v)a(A,B)^{(N)}f\|^p\ud u\ud v \\
  &= \frac{M_A^pM_B^p}{(2M)^{2d}}\int_{[-M,M]^2} \Big\| \frac1{(2\pi)^d} \int_{\R^{2d}}
     \wh a^{(N)}(y,\eta )U(-u,-v)U(y,\eta )f\ud y\ud \eta \,\Big\|^p\ud u \ud v  \\
  &=    \frac{M_A^pM_B^p}{(2M)^{2d}}\int_{[-M,M]^2}\! \Big\|  \frac1{(2\pi)^d}\!\int_{\R^{2d}}
     e^{\frac12i(u\eta-yv)}\wh a^{(N)}(y,\eta )U(y-u,\eta -v)f\ud y\ud \eta \,\Big\|^p\!\!\ud u\ud v  \\
\intertext{using \eqref{eq:sigma}. 
Also $\one_{[-M-N,M+N]^2 }(y-u,\eta-v)=1$ if
  $(u,v)\in[-M,M]^2$ and $(y,\eta )\in[-N,N]^2$, so that with
  $\chi_{M+N}:=\one_{[-M-N,M+N]^2}$ 
the last expression can be rewritten as}
  &= \frac{M_A^pM_B^p}{(2M)^{2d}}\int_{[-M,M]^2}\Bigl\n  \frac1{(2\pi)^d}\int_{\R^{2d}}
e^{\frac12i(u\eta-yv)}\wh a^{(N)}(y,\eta )
\\ & \qquad\qquad\qquad
\times \bigl[\chi_{M+N}(y-u,\eta-v)U(y-u,\eta -v)f\bigr]\ud y\ud \eta \, \Bigr\n^p\ud u\ud v \\
  &=\frac{M_A^pM_B^p}{(2M)^{2d}}\int_{[-M,M]^2}
    \| C_{\wh a^{(N)}} [\chi_{M+N}(\cdot,\cdot) U(\cdot,\cdot)f ](u,v)\|^p\ud u\ud v \\
  &\leq \frac{M_A^pM_B^p}{(2M)^{2d}}
    \| C_{\wh a^{(N)}} [\chi_{M+N} Uf]\|_{L^p(\R^{2d};X)}^p \\
  &\stackrel{\rm(i)}{\leq }\frac{M_A^pM_B^p}{(2M)^{2d}}
    \|C_{\wh a^{(N)}} \|_{\calL(L^p(\R^{2d};X))}^p \int_{[-M-N,M+N]^2}  \|U(u,v)f\|^p\ud u\ud v \\
  &\leq \frac{M_A^pM_B^p}{(2M)^{2d}}\|C_{\wh a^{(N)}} \|_{\calL(L^p(\R^{2d};X))}^p
    \, (2(M+N))^{2d}M_A^pM_B^p\|f\|^p \\
  & \leq (1+\eps)^{2d} M_A^{2p}M_B^{2p} \|C_{\wh a^{(N)}} \|_{\calL(L^p(\R^{2d};X))}^p\,\|f\|^p.
\end{align*}
It follows that 
\begin{align*} \n a(A,B)f\n 
& \le \n a(A,B)_{(N)}f\n +\n a(A,B)^{(N)}f\n 
\\ & \le \eps \n f\n + (1+\eps)^{2d/p}M_A^{2}M_B^{2} \|C_{\wh a^{(N)}} \|_{\calL(L^p(\R^{2d};X))}  \n f\n.
\end{align*}
Letting $N\to \infty$ in this estimate, and then letting $\eps\downarrow 0$, the desired estimate in (1) is obtained.

\smallskip
Part (2) is proved in exactly the same way.
We replace $a$ by $\sum_{n=1}^N \eps_n a_{j_n}$
(where $j_1,\dots,j_N\in J$ and $(\eps_n)_{n=1}^N$ is a Rademacher sequence) and instead of using one fixed 
$f$ we use a sequence $(f_n)_{n=1}^N$ to build Rademacher sums; 
instead of estimating with operator norms in (i), we estimate with $R$-bounds.
The same reasoning applies to part (3).
\end{proof}

The next lemma expresses the twisted convolution $C_{\wh a}$ in terms of the standard pair on $L^2(\R^{2d})$:

\begin{lemma} 
\label{lem:twiststand}
Let
 $((Q_1,Q_2), (P_1,P_2))$ be the standard pair of dimension $2d$ on $L^2(\R^{2d}),$ i.e.,
\begin{align*}
Q_{1,j}f(x,\xi) = x_jf(x,\xi), \quad & Q_{2,j}f(x,\xi) = \xi_jf(x,\xi), \\ 
P_{1,j}f(x,\xi) = \frac1i \frac{\partial f}{\partial x_j}(x,\xi),\quad   &P_{2,j}f(x,\xi) = \frac1i \frac{\partial f}{\partial \xi_j}(x,\xi),
\end{align*}
for $1\le j\le d$.
 The pair $(-\frac12 Q_2-P_1,\frac12Q_1-P_2)$ is a Weyl pair of dimension $d$ on $L^2(\R^{2d})$, 
 and for all $a\in \mathscr{S}(\R^{2d})$ we have 
$$ C_{\wh a} = a\Bigl(-\frac12 Q_2-P_1,\frac12Q_1-P_2\Bigr).$$ 
\end{lemma}

\begin{proof}
The proof of the first assertion is immediate.
For all $a\in \mathscr{S}(\R^{2d})$ and $g\in L^2(\R^{2d})$,
\begin{align*}
\ & a\Bigl(-\frac12 Q_2-P_1,\frac12Q_1-P_2\Bigr)g(x,\xi) 
\\ & \qquad =\frac1{(2\pi)^d}\int_{\R^{2d}} \wh a(u,v)e^{i(u(-\frac12Q_2-P_1)+v(\frac12 Q_1-P_2))}g(x,\xi) \ud u\ud v 
\\ & \qquad =\frac1{(2\pi)^d}\int_{\R^{2d}} \wh a(u,v)e^{\frac12iuv}e^{iu(-\frac12Q_2-P_1)}e^{iv(\frac12Q_1-P_2)}g(x,\xi) \ud u\ud v   
\\ & \qquad = \frac1{(2\pi)^d}\int_{\R^{2d}} \wh a(u,v)e^{\frac12iuv}e^{-\frac1{2}iuQ_2} e^{-iuP_1}e^{\frac1{2}iv Q_1} e^{-iv P_2}g(x,\xi) \ud u\ud v 
\\ & \qquad = \frac1{(2\pi)^d}\int_{\R^{2d}} \wh a(u,v) e^{\frac1{2}iv Q_1}e^{-\frac1{2}iuQ_2} e^{-iuP_1} e^{-iv P_2}g(x,\xi) \ud u\ud v 
\\ & \qquad = \frac1{(2\pi)^d}\int_{\R^{2d}} \wh a(u,v)e^{\frac12i(v x-\xi u)}g(x-u,\xi-v) \ud u\ud v 
\\ & \qquad = C_{\wh a}g(x,\xi).
\end{align*} 
\end{proof}

In the setting of the lemma, 
by the Stone--von Neumann theorem (see \cite[Theorem 14.8]{Hall}),
 there exist a countable index set $L$ and an orthogonal direct sum decomposition 
 $$ L^2(\R^{2d}) = \bigoplus_{\ell\in L} H_{\ell},$$
 as well as unitary operators $U_\ell: H_\ell \to L^2(\R^d)$, such that for all $\ell\in L$ the following assertions hold:
 \begin{enumerate}
  \item $H_{\ell}$ is invariant under each of the groups $e^{it(-\frac12 Q_{2,j}-P_{1,j})}$ and
  $e^{it(\frac12Q_{1,j}-P_{2,j})}$;
  \item $U_\ell$ establishes a unitary equivalence of these groups on $H_\ell$ with the groups $e^{itQ_j}$ and $e^{itP_j}$
  on $L^2(\R^d)$, where $(Q,P)$ is the standard pair on $L^2(\R^d)$. 
 \end{enumerate}
As a direct consequence we obtain that, for all $\ell\in L$ and $a\in \mathscr{S}(\R^{2d})$:
\begin{enumerate}
  \item[(1)$'$] $H_{\ell}$ is invariant under $a(-\frac12 Q_{2}-P_{1},\frac12Q_{1}-P_{2})$;
  \item[(2)$'$] $U_\ell$ establishes a unitary equivalence of the restriction of $a(-\frac12 Q_{2}-P_{1},\frac12Q_{1}-P_{2})$
to  $H_\ell$ and the operator $a(Q,P)$ on $L^2(\R^d)$. 
 \end{enumerate}
As a result we obtain
\begin{equation}\label{eq:scalar}
\begin{aligned}
\n C_{\wh a}\n_{\calL(L^2(\R^{2d}))} 
& = \Big\n a(-\frac12 Q_{2}-P_{1},\frac12Q_{1}-P_{2})\Big\n_{\calL(L^2(\R^{2d}))} 
\\ & = \sup_{\ell\in L} \Big\n a(-\frac12 Q_{2}-P_{1},\frac12Q_{1}-P_{2})\Big|_{H_\ell}\Bigr\n_{\calL(H_\ell)} 
 = \n a(Q,P)\n_{\calL(L^2(\R^{d}))}.
\end{aligned}
\end{equation}

\begin{remark}
In the next section we address the problem of estimating the norm of $C_{\wh a}$ in the vector-valued setting. Here we wish to point out the general fact that 
the identity \eqref{eq:scalar} has a simple, albeit not very useful (cf. the concluding remark at the end of the section), vector-valued extension in terms of spaces of $\gamma$-radonifying operators. These are defined  as follows (comprehensive treatments are given in \cite{HNVW2, Nee}). 
Let $H$ be a Hilbert space with inner product $(\cdot|\cdot)$ and $X$ be a Banach space. Every finite rank operator $T:H\to X$ can be represented as
$$ Th = \sum_{n=1}^N (h|h_n)x_n$$
for some orthonormal sequence $(h_n)_{n=1}^N$ in $X$, and some sequence $(x_n)_{n=1}^N$ in $X$.
For such operators $T$ we define
$$ \n T\n_{\gamma(H,X)}^2 := \E \Big\n \sum_{n=1}^N \gamma_n x_n\Big\n^2,
$$
where $(\gamma_n)_{n=1}^N$ is a sequence of independent standard normal random variables (taken real-valued if the scalar field is $\R$ and complex-valued if the scalar field is $\C$; once again, one could insist on using real-valued standard normal variables at the expense of different constants). It is easy to see that this gives a well-defined norm on the space of finite rank operators from $H$ to $X$. Its completion is denoted by $\gamma(H,X)$.

If $X$ is a Hilbert space, then $\gamma(H,X)$ is isometric in a natural way to the space of Hilbert--Schmidt operators from $H$ to $X$, and if $X = L^p(M,\mu)$ with $1\le p<\infty$, then one has a natural isomorphism of Banach spaces
$$\gamma(H, L^p(M,\mu)) \simeq L^p(M,\mu;H).$$

It is not hard to see (see
\cite[Theorem 9.6.1]{HNVW2}) that if $S: H\to H$ is a bounded operator, then the mapping $ h \otimes x \mapsto Sh \otimes x$
uniquely extends to a bounded operator $\wt S\in \calL(\gamma(H^*,X))$ of the same norm. Here,
$H^*$ is the Banach space dual of $H$. Applying this construction to the twisted convolutions $C_a$ and the operators $a(Q,P)$ with $a\in \mathscr{S}(\R^{2d})$, viewed as a bounded operators on the Hilbert spaces 
$L^2(\R^{2d})$ and $L^2(\R^d)$ respectively, and identifying the duals of these spaces with the spaces themselves via the duality $\lb f,g\rb = \int fg$ (no conjugation here),
we obtain well-defined extensions of these operators
to bounded operators on $\gamma(L^2(\R^{2d}),X)$ and $\gamma(L^2(\R^{d}),X)$ of the same norms.
Thus \eqref{eq:scalar} self-improves to
$$ \n C_{\wh a}\n_{\calL(\gamma(L^2(\R^{2d}),X))} =  \n a(Q,P)\n_{\calL(\gamma(L^2(\R^{d}),X))}.$$

This identity suggests that we could try to bound the Weyl calculus in terms of the
$\gamma(L^2(\R^{2d}),X))$-norm
of the twisted convolution. This is possible under a $\gamma$-boundedness assumption:
\end{remark}

\begin{proposition}[$\gamma$-Transference]\label{prop:g-transf} 
Let $(A,B)$ be a Weyl pair of dimension $d$ on a Banach space $X$.
If the set
$$ \{e^{i(uA+vB)}:\, (u,v)\in \R^{2d}\}$$
is $\gamma$-bounded, with $\gamma$-bound $\Gamma$, then
for all $a\in \mathscr{S}(\R^{2d})$ we have 
\begin{align*}
\n a(A,B)\n \le \Gamma^2\n C_{\wh a} \n_{\calL(\gamma(L^2(\R^{2d}),X))}.
\end{align*}
\end{proposition}

Similar versions of parts (2) and (3) of Proposition \ref{prop:transf} hold. The proof is a routine adaptation of the proof of Proposition \ref{prop:transf}. 
As a corollary we obtain that, if the set $\{e^{i(uA+vB)}: \, (u,v)\in\R^d\}$ is $\gamma$-bounded, with $\gamma$-bound $\Gamma$, 
then for all $a\in \mathscr{S}(\R^{2d})$ we have
$$ \n a(A,B)\n \le \Gamma^2\n a(Q,P)\n_{\calL(\gamma(L^2(\R^{d}),X))}.$$
Admittedly, this result is unlikely to be useful: for the standard pair, the $\gamma$-bounded\-ness assumption is satisfied only for $p=2$ (by \cite[Proposition 8.1.16]{HNVW2}).

\section{$R$-Sectoriality of $L$}\label{sec:untwist}

To apply the transference theory from Section \ref{sec:transf}, ideally one needs to bound twisted convolution operators $C_{\wh a}$ acting on the Bochner spaces $L^{p}(\R^{2d};X)$ in terms of the norm of $a(Q,P)$ for the standard pair, i.e., one needs a vector-valued extension of \eqref{eq:scalar}. We do not know how to do this in general. The $L^p$-theory in the scalar-valued case, considered by Mauceri in \cite{mauceri80}, is already quite subtle and depends on Hilbert space-specific techniques to treat the $p=2$ case. Extending his theory to UMD-valued functions would be interesting in itself (for the new techniques that need to be developed) and would lead to general estimates for the Weyl calculus of Weyl pairs.
Here, we just focus on those twisted convolutions needed to study the semigroup generated by $-L=\frac12d-\frac12(A^{2}+B^{2})$. The symbols $a$ involved are such that $C_{\wh a}$ can be effectively ``untwisted". 
The main aim of this section is to prove the following result:

\begin{theorem}[$R$-Sectoriality]\label{thm:L-R-sect} 
Let $(A,B)$ be a Weyl pair on a UMD Banach lattice $X$. Then for all $\theta\in (0,\pi)$ the operator $L = \frac12(A^{2}+B^{2}) - \frac12d$ is $R$-sectorial of angle $\theta$.  Moreover, the set
$
\{(\frac{\pi}{2}-\theta)^{2d} \exp(-zL) \;;\; |arg(z)|<\theta\}
$
is R-bounded, with R-bound independent of $\theta$.
\end{theorem}

The only place in the proof where we use the lattice structure of $X$ is in 
the following lemma, which reduces the task of proving $R$-boundedness of twisted convolutions to proving $R$-boundedness of (standard) convolutions. 

 \begin{lemma}\label{lem:untwist-Rbd} 
 Let $X$ be a Banach lattice with finite cotype and let $1\le p<\infty$.
Suppose $(a_j)_{j\in J}$ and $(b_j)_{j\in J}$ are families of functions in $\mathcal{S}(\R^{2d})$ satisfying
$$
|{a}_{j}(y,\eta)| \leq |{b}_{j}(y,\eta)| \quad \forall y,\eta \in \R^{d},\ j\in J.
$$
If the family of (standard) convolution operators $(\mathcal{C}_{|b_j|})_{j\in J}$ on $L^p(\R^{2d};X)$ is $R$-bounded, with $R_p$-bound $\mathscr{R}_p$, 
then also the family $(C_{{a_j}})_{j\in J}$ on $L^p(\R^{2d};X)$ is $R$-bounded, with $R_p$-bound $\lesssim_{p,q,X} C^p\mathscr{R}_p^p$.
\end{lemma}
\begin{proof}
Since $L^{p}(\R^{2d};X)$ has finite cotype (see \cite[Proposition 7.1.4]{HNVW2}),
we may use the Khintchine--Maurey Theorem (see \cite[Theorem 7.2.13]{HNVW2}) to pass from Radem\-acher sums to square functions. If 
$j_1,\dots,j_N\in J$ and $g_1\dots,g_N\in L^p(\R^{2d};X)$ are given
and $(\eps_n)_{n=1}^N$ is a Rademacher sequence, we thus obtain
\begin{align*}
\ & \E \Big\n \sum_{n=1}^N \eps_n C_{{a_{j_n}}} g_{j_n} \Big\n _{L^p(\R^{2d};X)}^p
\\ & \qquad \eqsim \Big\n \Big(\sum_{n=1}^N |C_{{a_{j_n}}} g_{j_n}|^{2} \Big)^{1/2} \Big\n_{L^p(\R^{2d};X)}^p\\
& \qquad  = \int_{\R^{2d}} \Big\n \Big(\sum_{n=1}^N \Big|\int_{\R^{2d}} e^{\frac12i(x\eta - y\xi)}
{a_{j_n}}(y,\eta) g_{j_n}(x-y,\xi-\eta)\ud y\ud\eta\Big|^{2}\Big)^{1/2} \Big\n^p\ud x\ud \xi \\
& \qquad  \leq \int_{\R^{2d}} \Big\n \Big(\sum_{n=1}^N \Big(\int_{\R^{2d}} |{b_{j_n}}(y,\eta)| |g_{j_n}(x-y,\xi-\eta)|\ud y\ud\eta\Big)^{2}\Big)^{1/2} \Big\n^p\ud x\ud \xi \\
& \qquad   = \Big\n \Big(\sum_{n=1}^N (\mathcal{C}_{|b_{j_n}|}|g_{j_n}|)^{2}\Big)^{1/2} \Big\n _{L^p(\R^{2d};X)}^p
\\ & \qquad \eqsim 
 \E \Big\n \sum_{n=1}^N \eps_n \mathcal{C}_{|b_{j_n}|} |g_{j_n}| \Big\n _{L^p(\R^{2d};X)}^p
\\ & \qquad \le \mathscr{R}_p^p \E \Big\n \sum_{n=1}^N \eps_n |g_{j_n}| \Big\n _{L^p(\R^{2d};X)}^p
 \\ & \qquad \eqsim  \mathscr{R}_p^p \Big\n \Bigl(\sum_{n=1}^N  |g_{j_n}|^2\Bigr)^{1/2} \Big\n _{L^p(\R^{2d};X)}^p
 \\ & \qquad \eqsim \mathscr{R}_p^p \E \Big\n \sum_{n=1}^N \eps_n g_{j_n} \Big\n _{L^p(\R^{2d};X)}^p
\end{align*}
with constants depending only on $p$ and $X$.
\end{proof}

We now consider the kernels relevant to our applications, namely the Fourier transforms of
$$a_{z}(x,\xi) = (1+\lambda_{z})^{d} e^{-\lambda_{z}(|x|^{2}+|\xi|^{2})}$$
with $\lambda_{z} = \frac{1-e^{-z}}{1+e^{-z}}$ for $z \in \C$ such that $\Re z>0$.
We need an elementary lemma. 

\begin{lemma}\label{lem:theta} 
For all $0<\theta<\frac12\pi$ and non-zero $z\in \C$ satisfying $|\arg z|\le \theta$
 we have 
$$
(\tfrac12\pi  - \theta) \lesssim \cos(\arg{\lambda_{z}}), \quad
|\lambda_{z}| \lesssim (\tfrac12\pi  - \theta)^{-1},
$$
with constants independent of $\theta$ and $z$.
\end{lemma}
\begin{proof}
It suffices to prove the inequalities for non-zero $z\in \C$ satisfying $|\arg z| = \theta$.
Writing $z = r(\cos\theta+i\sin\theta)$ and computing the real and imaginary parts
of $\frac{1-e^{-z}}{1+e^{-z}}$ in terms of $r$ and $\theta$, one readily finds
that if $|\arg z| < \theta$, then 
$$ \tan (\arg(\frac{1-e^{-z}}{1+e^{-z}})) < \frac1{\cos\theta} $$
and consequently
$$ \frac1{\cos(\arg(\frac{1-e^{-z}}{1+e^{-z}})) }< (1+\frac1{\cos^2\theta})^{1/2}
\le 1+ \frac1{\cos\theta}\lesssim \frac1{\tfrac12\pi  - \theta}.$$
Similar elementary estimates show that if $|\arg z| < \theta$, then 
$$ |\lambda_z| = \Big|\frac{1-e^{-z}}{1+e^{-z}}\Big| \lesssim \frac1{\cos\theta}
\lesssim \frac1{(\tfrac12\pi  - \theta)}.$$
\end{proof}

This lemma is used to prove the following $R$-boundedness result.

\begin{proposition}
\label{prop:convR}
Let $X$ be a  UMD  Banach lattice, and let $1<p<\infty$.
There exists a constant $M\ge 0$,  depending only on $p$ and $X$, such that for all
$\theta \in (0,\frac12\pi )$
the family $\{C_{\wh a_z } \;;\; z\not=0,\, |\arg z| < \theta\}$
is $R$-bounded in $\calL(L^p(\R^{2d};X))$, with constant 
$$\mathscr{R}(\{C_{\wh a_z } \;;\; z\not=0, \,|\arg z| < \theta\}) \le \frac{M}{(\frac{1}{2}\pi-\theta)^{2d}} .$$
\end{proposition}

\begin{proof}
The space $X$, being UMD, has finite cotype (see \cite[Proposition 7.3.15]{HNVW2}).
Fix $\theta \in (0,\frac12\pi )$ and let $z\in \C$ be a non-zero element such that 
$|\arg z|< \theta$ for all $k=1, \dots , N$.
Writing $t = 1/\Re(1/\lambda_{z}) = |\lambda_{z}|/\cos(\arg{\lambda_{z}})$,
for all $y,\eta \in \R^{d}$ we have
\begin{align*}
|\wh a_z (y,\eta)| &\eqsim |1+\lambda_{z}|^{d} 
|\lambda_{z}|^{-d}
e^{-\cos(\arg{\lambda_{z}})(|y|^{2}+|\eta|^{2})/4|\lambda_{z}|} \\
& \lesssim 
|1+\lambda_{z}|^{d} (\cos(\arg{\lambda_{z}}))^{-d} 
t^{-d} e^{-(|y|^{2}+|\eta|^{2})/4t}
\end{align*}
with constants depending only on $p$ and $X$.
Hence by Lemma \ref{lem:theta},
\begin{align*}
(\tfrac12\pi -\theta)^{2d}|\wh a_z (y,\eta)|  \lesssim 
t^{-d} e^{-\frac{(|y|^{2}+|\eta|^{2})}{4t}}=:b_{t}(y,\eta),
\end{align*}
with constants depending only on $p$ and $X$.
Hence, by Lemma \ref{lem:untwist-Rbd}, 
$$
\mathcal{R}(\{(\tfrac12\pi -\theta)^{2d} C_{\wh a_z } \;;\; z\not=0,\, |\arg z| < \theta\})
\lesssim 
\mathcal{R}(\{\mathcal{C}_{b_{t}} \;;\; t>0\}),
$$
with a constant depending only on $p$ and $X$.
Noting that $\mathcal{C}_{b_{t}}$ is a constant multiple of  $\exp(t\Delta \otimes I_{X})$, the
$R$-boundedness of the family $\{\mathcal{C}_{b_{t}} \;;\; t>0\}$ follows from
the fact that $-\Delta \otimes I_{X}$ has a bounded $H^{\infty}$-calculus
on $L^p(\R^{2d};X)$ \cite[Theorem 10.2.25]{HNVW2}. 
\end{proof}

\begin{proof}[Proof of Theorem \ref{thm:L-R-sect}] 
Let $(P(t))_{t\ge 0}$ be the analytic $C_0$-semigroup generated by $-L 
= \frac12d-\frac12(A^2+B^2)$. Fix $\theta \in (0, \frac12{\pi})$.
By  \cite[Proposition 10.3.3]{HNVW2} it suffices to show that the set
$$
V_\theta :=\bigl\{P(z):\,z\not=0,\,|\arg z| < \theta\bigr\}
$$
is $R$-bounded with
$$\mathscr{R}(V_\theta)\lesssim_{p,X} \frac{1}{(\frac12{\pi}-\theta)^{2d}}.$$
By Theorem \ref{thm:A2B2} and Proposition \ref{prop:transf}, for this it suffices to show that the set
$$
V'_\theta :=\bigl\{C_{\wh a_z }:\,z\not=0,\,|\arg z| < \theta\bigr\}
$$
is $R$-bounded with
$$\mathscr{R}(V'_\theta)\lesssim_{p,X} \frac{1}{(\frac12{\pi}-\theta)^{2d}}.$$
This has been done in Proposition \ref{prop:convR}.
\end{proof}

\section{Functional calculus of $L$}\label{sec:Hinfty}

In this section we prove that boundedness of the Weyl calculus of a Weyl pair $(A,B)$ implies a spectral multiplier theorem for the operator
$L= \frac12(A^2+B^2)-\frac12d$, acting on a UMD lattice $X$. This is done by applying the theory developed in \cite{KalWei} to obtain a holomorphic functional calculus of angle zero from square function estimates and appropriate $R$-sectoriality bounds. The precise form of the latter then allows us to apply the theory developed in \cite{KriegW} to extend this holomorphic functional calculus to a full H\"ormander type spectral multiplier theorem.

\begin{theorem}[Bounded $H^\infty$-calculus]\label{thm:HinftyL}
 If $(A,B)$ is a Weyl pair of dimension $d$ on a UMD lattice $X$ with a bounded Weyl calculus of type $0$, then
 $L:=\frac12(A^2+B^2)-\frac12d$ has a bounded 
 $H^\infty(\Sigma_\theta)$-calculus for all $\theta\in (0,\pi)$.
 \end{theorem}
 
For the proof of the theorem we need two lemmas. The first provides an expression for
 derivatives of the exponentials in the Weyl calculus representation formula of the operators $e^{-tL}$ of Theorem \ref{thm:A2B2}.
 
 \begin{lemma}\label{lem:pol} 
 For all multi-indices $\alpha,\beta\in\N^d$ there is a polynomial 
 $p_{\alpha,\beta}$ of degree $(\alpha,\beta)$ in the variables $(x,\xi)\in \R^{2d}$ such that 
 for all $\lambda>0$ we have
 $$\partial_{\xi} ^{\alpha} \partial_{x} ^{\beta}e^{- \lambda(| x|^2+|\xi|^2)}= 
\sqrt{\lambda}^{|\alpha|+|\beta|} p_{\alpha,\beta}(\sqrt{\lambda} x,\sqrt{\lambda}\xi)e^{-\lambda(|x|^2+|\xi|^2)}.$$
 \end{lemma}
\begin{proof}
If $p$ is a polynomial in $2d$ variables $x=(x_1,\dots,x_d)$ and $\xi = (\xi_1,\dots,\xi_d)$,
of degree $\gamma=(\gamma_1,\dots,\gamma_d)$ in $x$ and $\delta=(\delta_1,\dots,\delta_d)$ in $\xi$, then for any $\lambda>0$, 
\begin{align*} 
 & 
\partial_{x_j} [p( \sqrt{\lambda} x, \sqrt{\lambda} \xi)e^{- \lambda(| x|^2+|\xi|^2)}]
\\ & \qquad \qquad = [ \sqrt{\lambda} (\partial_{x_j}p)( \sqrt{\lambda} x, \sqrt{\lambda} \xi) -2 \lambda x_j p(\sqrt{\lambda} x, \sqrt{\lambda} \xi)]e^{-\lambda (|x|^2+|\xi|^2)}
\\ &  \qquad \qquad  = \sqrt{\lambda} q(\sqrt{\lambda} x, \sqrt{\lambda} \xi)e^{-\lambda (|x|^2+|\xi|^2)},
\end{align*}
where $q$ is polynomial of degree $(\gamma_1,\dots,\gamma_{j-1},\gamma_j+1,\gamma_{j+1},\dots,\gamma_d)$ in $x$ and of degree $\delta$ in $\xi$.
A similar identity holds for the partial derivatives with respect to $\xi_j$, which add one to the degree in the variable $\xi_j$.  
The lemma now follows by induction on $\alpha$ and $\beta$.
\end{proof}

As an application of the preceding lemma, the next lemma provides a uniform bound on the derivatives of certain signed sums of exponentials which will be used later to prove that certain related sums belong to the symbol class $S^{0}$ uniformly.

As before, let $\lambda_{t} = \frac{1-e^{-t}}{1+e^{-t}}$ for $t>0$.

\begin{lemma}\label{lem:exp-symb}
 For all multi-indices $\alpha,\beta\in\N^d$ such that $|\alpha|+|\beta|\not=0$
 the functions
 $$ \kappa_{k,\epsilon,s}(x,\xi) :=\sum_{j=1}^k \epsilon_j \exp(-\lambda_{2^{-j} s}(|x|^2+|\xi|^2))$$
 satisfy
 $$\sup_{(x,\xi)\in \R^{2d}}\, \lb\xi\rb^{|\alpha|}|\partial_\xi^{\alpha}\partial_x^\beta \kappa_{k,\epsilon,s}(x,\xi)| < \infty
 $$ 
 uniformly with respect to 
 $k\ge 1$, $\epsilon = (\epsilon_j)_{j=1}^k\in \{\pm 1\}^k$, and $s\in [1,2]$.
\end{lemma}

\begin{proof}
Let us set $\mu_{j,s} = \lambda_{2^{-j}s}$ for brevity.
Given any two multi-indices $\alpha,\beta \in  \N^{d}$ 
such that $|\alpha|+|\beta|\not=0$ we may estimate, using Lemma \ref{lem:pol},
\begin{align*}
\ & \lb\xi\rb^{|\alpha|}|\partial_\xi^{\alpha}\partial_x^\beta \kappa_{k,\epsilon,s}(x,\xi)| 
\\ & \qquad \le
\lb\xi\rb^{|\alpha|}
\sum_{j=1}^k 
\sqrt{\mu_{j,s}}^{|\alpha|+|\beta|} 
|p_{\alpha,\beta}(\sqrt{\mu_{j,s}}x,\sqrt{\mu_{j,s}}\xi)|
\exp(-\mu_{j,s}(|x|^2+|\xi|^{2})).
\intertext{For $|\xi|\le 1$ we estimate the right-hand side by}
 & \qquad  \lesssim_\alpha
\sum_{j=1}^k 
\sqrt{\mu_{j,s}}^{|\alpha|+|\beta|} 
|p_{\alpha,\beta}(\sqrt{\mu_{j,s}}x,\sqrt{\mu_{j,s}}\xi)|
\exp(-\mu_{j,s}(|x|^2+|\xi|^{2}))
\\ & \qquad \lesssim_{\alpha,\beta}
\sum_{j=1} ^{k}\sqrt{\mu_{j,s}}^{|\alpha|+|\beta|} 
\\ & \qquad \lesssim_{\alpha,\beta} \sum_{j=1} ^{k} 2^{-\frac12j(|\alpha|+|\beta|)}
\\ & \qquad \lesssim_{\alpha,\beta} 1
\intertext{where we used that $\sup_{(x',\xi')\in \R^{2d}}|p_{\alpha,\beta}(x',\xi')|\exp(-(|x'|^2+|\xi'|^{2}))<\infty$ and $\mu_{j,s} \lesssim 2^{-j}$; while 
for $|\xi|>1$ we may estimate it by
          }
& \qquad \lesssim_{\alpha}
\sum_{j=1}^k (\sqrt{\mu_{j,s}}|\xi|)^{|\alpha|+|\beta|}
p_{\alpha,\beta}(\sqrt{\mu_{j,s}}x,\sqrt{\mu_{j,s}}\xi)|\exp(-\mu_{j,s}(|x|^2+|\xi|^{2}))
\\ & \qquad \lesssim_{\alpha,\beta} 1,
\end{align*}
where the last step follows by an application of \cite[Proposition H.2.3]{HNVW2}.
In all these estimates, 
the  constants are uniform in $k$, $\epsilon$, and $s$.
\end{proof}

 \begin{proof}[Proof of Theorem \ref{thm:HinftyL}]
By Theorem \ref{thm:L-R-sect} $L$ is $R$-sectorial of angle $\theta$ for any $\theta\in (0,\pi)$. Hence by
\cite[Theorem 10.4.9]{HNVW2} it suffices to show that
$$
\|f\|^2 \sim\underset{s \in [1,2]}{\sup}\sup_{N\in \N}\E \Big\|\sum  _{|j|\le N} \eps_{j}(\exp(-2^{j+1}sL) - \exp(-2^{j}sL))f\Big\|^2 \quad \forall f \in \Dom(L)\cap \Ran(L),
$$
where $(\eps_j)_{j\in\Z}$ is a Rademacher sequence,
noting that the function $z\mapsto \exp(-2z) - \exp(-z)$ belongs to $H^1(\Sigma_\theta)\cap
H^\infty(\Sigma_\theta)$ for each $\theta\in (0,\frac12\pi)$ 
using the notation of \cite[Chapter 10]{HNVW2}.
It actually suffices to prove the one-sided inequality
\begin{align}\label{eq:onesided}
\underset{s \in [1,2]}{\sup}\sup_{N\in\N}\E \Big\|\sum  _{|j|\le N} \eps_{j}(\exp(-2^{j+1}sL) - \exp(-2^{j}sL))f\Big\|^2
\lesssim \|f\|^2 \quad \forall f \in X,
\end{align}
since the reverse inequality (for $f\in \overline{\Ran(L)} = \overline{\Dom(L)\cap \Ran(L)}$) will then follow by duality as in the proof of \cite[Theorem 10.4.4(3)]{HNVW2}, noting that the pair of adjoint operators $(B^*,A^*)$ is a Weyl pair in the dual lattice $X^{*}$ (which is UMD by \cite[Proposition 4.2.17]{HNVW1} and \cite[Proposition 7.5.15]{HNVW2}).

Referring to the direct sum decomposition $X = \Ker(L) \oplus\overline{\Ran(L)}$ provided by Proposition \ref{prop:nullrange}, 
we will prove \eqref{eq:onesided} separately for $f\in \Ker(L)$ and $f\in\overline{\Ran(L)}$. 
For $f \in \Ker(L)$, \eqref{eq:onesided} is immediate from the fact that $\exp(-tL)f=f$. 
For $f\in \overline{\Ran(L)}$ we consider indices
$j\in \N$ and $j\in \Z\setminus\N$ separately. For $j\in \N$ and $f\in \Ran(L)$, say $f = Lg$, 
we use that $\exp(-tL)f = L\exp(-tL)g$ decays to $0$ exponentially as $t\to\infty$
by Corollary \ref{cor:spectral}. In combination with the triangle inequality and the contraction principle, this gives
\begin{align*}
\ & \underset{s \in [1,2]}{\sup}\sup_{N\in\N}\Bigl(\E \Big\| \sum_{j=0}^N  
\eps_{j}(\exp(-2^{j+1}sL) - \exp(-2^{j}sL))f\Big\|^2\Bigr)^{1/2} \\
& \qquad \le  
2\underset{s \in [1,2]}{\sup}\sup_{N\in\N}\Bigl(\E \Big\| \sum_{j=0}^N \eps_{j}\exp(-2^{j}sL)f\Big\|^2\Bigr)^{1/2}
\\ & \qquad \le 2
\underset{s \in [1,2]}{\sup}\sup_{N\in\N}\sum_{j=0}^{N} \|\exp(-2^{j}sL)f\|
\\ & \qquad \lesssim\|f\|.
\end{align*}
By continuity, this estimate extends to $f\in \overline{\Ran(L)}$.

Let $$\widetilde a_{t}(x,\xi) := a_{2t}(x,\xi)-a_{t}(x,\xi),$$ 
where as always $a_t(x,\xi) = (1+\lambda_{t})^{d}e^{-\lambda_{t}(|x|^{2}+|\xi|^{2})}
$ with $\lambda_t:= \frac{1-e^{-t}}{1+e^{-t}}$. In view of the cases already dealt with,  
the proof will be complete once we have shown that, for all
$f\in \overline{\Ran(L)}$,
\begin{equation}
\label{eq:sfe}
\sup_{s\in [1,2]}\sup_{N\ge 1}\mathbb{E}\Big\|\sum  _{j=1} ^{N} \eps_{j} \widetilde a_{2^{-j}s}(A,B)f\Big\|^2\lesssim \n f\n^2 \end{equation}
with a constant independent of $f$.
Set 
\begin{align*}\wt b_{t} :& = (1+\lambda_{2t})^{-d}a_{2t} - (1+\lambda_{t})^{-d}a_{t} 
\end{align*}
so that 
\begin{align*}
\wt a_t = \wt b_{t}
- ((1+\lambda_{2t})^{-d}-1)a_{2t} +  ((1+\lambda_{t})^{-d}a_{t}-1)a_t.
\end{align*}
We now take $t = 2^{-j}s$ and estimate each of the resulting three sums separately. 
Fix an integer $N \ge 1$. 
We first estimate
\begin{align*}
\sup_{s\in [1,2]}&\sup_{N\ge 1}\mathbb{E}\Big\|\sum  _{j=1} ^{N} \eps_{j} ((1+\lambda_{2^{-j+1}})^{-d}-1)a_{2^{-j+1}s}(A,B)f\Big\|^2 \\ &\lesssim 
\sup_{s\in [1,2]} \Bigl(\sum  _{j=1} ^{\infty} |(1+\lambda_{2^{-j+1}})^{-d}-1|
\|\exp(-2^{-j+1}sL)f \|\Bigr)^2 \\ &
\lesssim_d \Bigl(\sum  _{j=0} ^{\infty} 2^{-j}
\n f\n\Bigr)^2  \lesssim \n f\n^2 ,
\end{align*}
where we used the bound $(1+\lambda_{2^{-j}})^{-d}-1\lesssim_d 2^{-j}$ 
together with the uniform boundedness of the operators $\exp(-tL) = P(t)$. 
Similarly,
$$ \sup_{s\in [1,2]} \sup_{N\ge 1}\mathbb{E}\Big\|\sum_{j=1}^{N} \eps_{j} ((1+\lambda_{2^{-j}})^{-d}-1)a_{2^{-j}s}(A,B)f\Big\|^2 \lesssim \n f\n^2.$$
To prove \eqref{eq:sfe}, it therefore remains to show that
\begin{equation*}
\sup_{s\in [1,2]}\sup_{N\ge 1}\mathbb{E}\Big\|\sum_{j=1}^{N} \eps_{j} \wt b_{2^{-j}s}(A,B)f\Big\|^2\lesssim \n f\n^2 .
\end{equation*}
To this end we claim that the functions 
$$\wt\kappa_{N,\epsilon,s}:=\sum_{j=1}^N \epsilon_{j} \wt b_{2^{-j}s}$$
belong to the symbol class $S^0$, uniformly in $N$, $\eps\in \{\pm 1\}^N$, and $s\in [1,2]$.
Since by assumption $(A,B)$ has a bounded Weyl calculus of type $0$, this claim, once it has been proved, will prove the theorem.

We have
\begin{align*}
|\wt\kappa_{N,\epsilon,s}(x,\xi)| & 
\le  \sum  _{j=1} ^{N} |\wt b_{2^{-j}s}(x,\xi)|
\\ & =  \sum_{j=1} ^N 
\Bigl|\exp(-\lambda_{2^{-j+1}s}(|x|^{2}+|\xi|^{2}))-\exp(-\lambda_{2^{-j}s}(|x|^{2}+|\xi|^{2}))\Bigr| \\
& = \sum_{j=1} ^N 
\bigl(\exp(-\lambda_{2^{-j+1}s}(|x|^{2}+|\xi|^{2}))-\exp(-\lambda_{2^{-j}s}(|x|^{2}+|\xi|^{2}))\bigr) \\
& \leq 1
\end{align*}
using a telescoping argument in the last step.
In combination with Lemma \ref{lem:exp-symb} (which remains true if we replace the 
summation $\sum_{j=1}^N$ by $\sum_{j=0} ^{N-1}$), this proves the claim.
\end{proof}
 
\begin{remark}\label{rem:no-shift}
For the standard pair $(Q,P)$ on $L^p(\R^d;X)$ with $1<p<\infty$ and $X$ any UMD space, the operator $\frac12(Q^2+P^2)-\frac12d$ is
$R$-sectorial and has a bounded $H^\infty$-calculus for any angle $\theta\in (0,\pi)$. Following the lines of \cite{MaaNee}, this follows from the results of \cite{BCCFR} (see also \cite{AFT}). 
\end{remark}

Using the theory developed in \cite{KriegW}, we can extend the functional calculus of $L$ from $H^{\infty}(\Sigma_{\theta})$ to an appropriate H\"ormander class. We 
thus obtain a calculus in one of their $\mathcal{H}_{2} ^{\beta}$ classes. As pointed out in \cite[Remark 3.3]{KriegW}, these classes are slightly larger (but more complicated to define) than the standard H\"ormander classes of functions $f \in C^{m}[0,\infty)$  satisfying
$$
\underset{R>0}{\sup} \ R^{2k} \int _{\frac12 R} ^{2R} |f^{(k)}(t)|^{2} \,\frac{{\rm d}t}{R}<\infty,
$$
for all $k=0, \dots , m$. Note that the latter class contains all smooth functions
with compact support in $(0,\infty)$.

\begin{theorem}\label{thm:Horm}
 If $(A,B)$ is a Weyl pair on a UMD lattice $X$ with a bounded Weyl calculus of type $0$, then
 $L=\frac12(A^{2}+B^{2})-\frac12d$ has an $R$-bounded $\mathcal{H}_{2} ^{2d+\frac{1}{2}}$-H\"ormander calculus. 
\end{theorem} 

\begin{proof}
By Theorems  \ref{thm:L-R-sect}, \ref{thm:HinftyL}, the assumptions of \cite[Theorem 7.1]{KriegW} are satisfied (note that UMD lattices have the required property $(\alpha)$ by \cite[Theorem 7.5.20; see the Notes to this section for the terminology]{HNVW2}).
\end{proof}

\section{Open problems}
As explained in the introduction, this paper is mostly meant as a foundation for the development of pseudo-differential calculi in ``rough'' settings. We nonetheless think that the general theory of Weyl pairs presented here is also worth developing further in its own right. This would include solving the following problems: 

\begin{enumerate}
\item[\rm(1)]
To extend Mauceri's results on twisted convolutions \cite{mauceri80} to Bochner spaces $L^{p}(\R^{2d};X)$, where $X$ is UMD and $p\in (1,\infty)$. 
\end{enumerate}
An affirmative answer would automatically solve the next problem.  

\begin{enumerate}
\item[\rm(2)]  To extend Theorem \ref{thm:type0} to general Weyl pairs $(A,B)$.  

\end{enumerate}
As observed in Remark \ref{rem:no-shift},
for standard pairs on $L^p(\R^d;X)$ with $1<p<\infty$, the conclusions of
Theorems \ref{thm:L-R-sect}, \ref{thm:HinftyL}, and \ref{thm:Horm} hold for arbitrary UMD Banach spaces $X$. 
In the three theorems for general Weyl pairs, the lattice structure of $X$ was only used through the proof of Lemma \ref{lem:untwist-Rbd}.
Thus one may pose the following problem:

\begin{enumerate}
\item[\rm(3)]  To decide whether 
Theorems \ref{thm:L-R-sect}, \ref{thm:HinftyL}, and \ref{thm:Horm} hold for arbitrary UMD spaces $X$. 
\end{enumerate}
An affirmative answer to the first problem could possibly also solve the third problem, 
since it might pave the way for an alternative proof via transference.
For these three problems, studying the particular case where $X$ is a non-commutative $L^p$-space would be particularly interesting, yet potentially much simpler than the general case (thanks to the availability of both domination and extrapolation techniques).

\begin{enumerate}
\item[\rm(4)]  To prove an analogue of the Stone--von Neumann uniqueness theorem for Weyl pairs.
\end{enumerate}

\end{document}